\documentclass[journal,twoside,web]{ieeecolor}
\bstctlcite{bstctl:etal, bstctl:nodash, bstctl:simpurl}

\usepackage{generic}
\usepackage{color,amsmath,mathrsfs,amsfonts,amssymb,graphicx,epsfig}
 \usepackage{amsthm}
\usepackage{subcaption}
\usepackage{multirow,rotating,diagbox}
\usepackage{algorithm,algpseudocode,algorithmicx}
\usepackage{cite,url,framed,bm,balance,nicematrix}
\usepackage{stmaryrd,hyperref}
\let\labelindent\relax
\usepackage{enumitem}
\hypersetup{
    colorlinks=true,
    linkcolor=blue,
    citecolor=blue,
    pdftitle={Overleaf Example},
    pdfpagemode=FullScreen,
    }

\newtheorem{theorem}{Theorem}
\newtheorem{corollary}[theorem]{Corollary}
\newtheorem{definition}{Definition}

\newtheorem{proposition}{Proposition}
\newtheorem{remark}{Remark}

\usepackage{tikz}
\usetikzlibrary{calc,trees,positioning,arrows,chains,shapes.geometric,decorations.pathreplacing,decorations.pathmorphing,shapes, matrix,shapes.symbols}

\usepackage[export]{adjustbox}

\makeatletter
\newcommand{\pushright}[1]{\ifmeasuring@#1\else\omit\hfill$\displaystyle{#1}$\fi\ignorespaces}

\newcommand{\differential}{{\rm{d}}}

\newcommand{\hess}{{\mathrm{Hess}}}

\newcommand{{\dx}}{{{\rm d} x}}
\newcommand{{\dy}}{{{\rm d} y}}
\newcommand{{\dt}}{{{\rm d} t}}

\makeatletter
\newcommand{\dotminus}{\mathbin{\text{\@dotminus}}}

\newcommand{\@dotminus}{%
  \ooalign{\hidewidth\raise1ex\hbox{.}\hidewidth\cr$\m@th-$\cr}%
}

\allowdisplaybreaks

\def\BibTeX{{\rm B\kern-.05em{\sc i\kern-.025em b}\kern-.08em
    T\kern-.1667em\lower.7ex\hbox{E}\kern-.125emX}}
\begin{document}
\bstctlcite{IEEE_b:BSTcontrol}
\title{Nonlinear Non-Gaussian Density Steering with Input and Noise Channel Mismatch: Sinkhorn with Memory for Solving the Control-affine Schr\"{o}dinger Bridge Problem}
\author{Georgiy A. Bondar, Asmaa Eldesoukey, Yongxin Chen, Abhishek Halder
\thanks{Georgiy A. Bondar is with the Department of Applied Mathematics, University of California Santa Cruz, CA 95064, USA, {\tt\footnotesize gbondar@ucsc.edu}.}
\thanks{Abhishek Halder and Asmaa Eldesoukey are with the Department of Aerospace Engineering, Iowa State University, Ames, IA 50011, USA, {\tt\footnotesize{\{ahalder,asmaae\}@iastate.edu}}.}
\thanks{Yongxin Chen is with the School of Aerospace Engineering, Georgia Institute of Technology, Atlanta, GA 30332, USA, {\tt\footnotesize{ychen3148@gatech.edu}}.}
\thanks{This research was partially supported by NSF awards 2111688, 2450377, 2450378.}
}

\maketitle

\begin{abstract}
Solutions to the Schr\"{o}dinger bridge problem and its generalizations yield feedback control policies for optimal density steering over a controlled diffusion. To numerically compute the same, the dynamic Sinkhorn recursion has become a standard approach. The mathematical engine behind this approach is the Hopf-Cole transform that recasts the conditions for optimality into a system of boundary-coupled linear PDEs. Recent works pointed out that for the control-affine Schr\"{o}dinger bridge problem, this exact linearity via Hopf-Cole transform, and thus the standard Sinkhorn recursion, apply only if the control and noise channels are proportional. When the channels do not match, the Hopf-Cole-transformed PDEs remain nonlinear, and no algorithm is available to solve the same. We advance the state-of-the-art by designing a Sinkhorn recursion with memory that leverages the structure of these nonlinear PDEs, and demonstrate how it solves the control-affine Schr\"{o}dinger bridge problem with input and noise channel mismatch. We prove the local stability of the proposed algorithm.
\end{abstract}



\section{Introduction}\label{sec:Intro}
The purpose of this work is to investigate the control-affine Schr\"{o}dinger bridge problem (CASBP) \cite{teter2025hopf,11208710} from a computational perspective. The CASBP is a stochastic optimal control problem of the form
\begin{subequations}
\begin{align}
&\underset{(\rho^{\bm{u}},\bm{u})\in\mathcal{P}_{01}\times\mathcal{U}}{\arg\inf}\quad\int_{t_0}^{t_1}\mathbb{E}_{\rho^{\bm{u}}}\left[q(t,\bm{x}^{u}) + \frac{1}{2}\|\bm{u}\|_2^2\right]\:\differential t\label{CASBPobj}\\
&\text{subject to}\nonumber\\
&\mathrm{d} \boldsymbol{x}^{\boldsymbol{u}}=\left(\boldsymbol{f}\left(t, \boldsymbol{x}^{\boldsymbol{u}}\right)+\boldsymbol{g}\left(t, \boldsymbol{x}^{\boldsymbol{u}}\right) \boldsymbol{u}\right) \mathrm{d} t+\boldsymbol{\sigma}\left(t, \boldsymbol{x}^{\boldsymbol{u}}\right) \mathrm{d} \boldsymbol{w}\label{CASBPsde},
\end{align}
\label{CASBP}
\end{subequations}%
where $0\leq t_0 < t_1 < \infty$, the control input $\bm{u}\in\mathbb{R}^{m}$, the controlled state $\bm{x}^{\bm{u}}\in\mathbb{R}^{n}$, the standard Wiener process $\bm{w}\in\mathbb{R}^{p}$, and for given probability density functions (PDFs) $\rho_0,\rho_1$ with \emph{finite second moments}, the feasible sets
\begin{subequations}
\begin{align}
\!\! \mathcal{P}_{01} &:= \{t \mapsto \rho(t, \cdot)\in\mathcal{C}^{1}([t_0,t_1]) \mid \rho \geq 0, \int \!\rho(t,\cdot) = 1\nonumber\\
&\quad\forall t \in[t_0,t_1],\rho\left(t_0, \cdot\right)=\rho_0(\cdot), \rho\left(t_1, \cdot\right)=\rho_1(\cdot)\},\label{defP01}\\
\mathcal{U} &:= \left\{\boldsymbol{u}:\left[t_0, t_1\right] \times \mathbb{R}^n \mapsto \mathbb{R}^m \mid\|\boldsymbol{u}\|_2^2<\infty\right\}.\label{defFeasibleInputSet}
\end{align}
\label{FeasibleSet}
\end{subequations}%
The state cost $q\geq 0$ is assumed to be bounded and integrable with respect to the measure $\rho^{\bm{u}}\differential\bm{x}^{\bm{u}}$ for all $(\rho^{\bm{u}},\bm{u})\in\mathcal{P}_{01}\times\mathcal{U}$. 

So the design objective is to guide the state from the initial PDF $\rho_0$ to the terminal PDF $\rho_1$ over a given finite time horizon $[t_0,t_1]$ via finite energy controls while minimizing an average additive cost (sum of a state cost $q\geq 0$ and a control cost $\frac{1}{2}\|\bm{u}\|_2^2$).

As standard, the positive semidefinite matrix field $\bm{\Sigma}:=\bm{\sigma\sigma}^{\top}\succeq\bm{0}$ is referred to as the \emph{diffusion tensor}. We make the following standard regularity assumptions on the drift, input, and noise coefficients $\bm{f},\bm{g},\bm{\sigma}$, respectively.
\begin{enumerate}[label=\textbf{A\arabic*}]
\item(\textbf{Non-explosion and Lipschitz coefficients}) There exist finite constants $c_1,c_2>0$ such that $\forall\bm{x},\bm{y}\in\mathbb{R}^n$, $\forall t\in[t_0,t_1]$, we have $\|\bm{f}(t,\bm{x})\|_2 + \|\bm{\sigma}\left(t,\bm{x}\right)\|_{2} 
\leq 
c_1\left(1 + \|\bm{x}\|_2\right)$, and
$\|\bm{f}(t,\bm{x}) - \bm{f}(t,\bm{y})\|_2 
\leq 
c_2 \|\bm{x}-\bm{y}\|_2$.    \label{A1}
\item (\textbf{Uniformly lower bounded diffusion tensor}) There exists a finite constant $c_3>0$ such that $\forall \bm{x}\in\mathbb{R}^{n}$, $\forall t\in[t_0,t_1]$, we have $\langle \bm{x}, \bm{\Sigma}(t,\bm{x})\bm{x}\rangle 
    \geq 
    c_3 \|\bm{x}\|_2^2$. \label{A2}
\item (\textbf{Uniformly upper bounded input coefficient}) There exists a finite constant $c_4>0$ such that $\forall \bm{x}\in\mathbb{R}^{n}$, $\forall t\in[t_0,t_1]$, we have $\|\bm{g}(t,\bm{x})\|_2< c_4$. \label{A3}
\end{enumerate}

The CASBP generalizes the classical SBP \cite{Sch31,Sch32} in that the latter is the following special case of \eqref{CASBP}: $$\bm{f}=\bm{0},\,\bm{g}=\bm{\sigma}=\bm{I},\,q=0.$$
Given problem data $\bm{f},\bm{g},\bm{\sigma},q,\rho_0,\rho_1$, the CASBP \eqref{CASBP}-\eqref{FeasibleSet} yields unique Markovian dynamic state feedback policy $\bm{u}\in\mathcal{U}$, see e.g., \cite{wakolbinger1990schrodinger}. This control policy guarantees that the controlled state statistics go from $\rho_0$ to $\rho_1$ over the specified time horizon $[t_0,t_1]$, in addition to minimizing the average cost \eqref{CASBPobj} subject to \eqref{CASBPsde}.

\subsubsection*{Channel mismatch} Our particular focus is on solving a generic CASBP with ``channel mismatch", by which we mean that the input coefficient $\bm{g}$ and the diffusion coefficient $\bm{\sigma}$ are not \emph{mutually related}. In particular, $\bm{gg}^{\top}\succeq\bm{0}$ is not proportional to the diffusion tensor $\bm{\Sigma}$. Stated differently, 
\begin{align}
\nexists\lambda>0\quad\text{such that}\quad \lambda\bm{gg}^{\top}-\bm{\Sigma}=\bm{0}.
\label{DifferentChannel}    
\end{align} 
This is the situation when the noise and the input do not act on the same subspace \cite[Sec. IV-B]{teter2025hopf}.

We clarify that the channel mismatch issue is pertinent for $n\geq 2$ states, i.e., when the size of the square matrices $\bm{gg}^{\top},\bm{\Sigma}$ are at least $2\times 2$. For $n=1$, the input and noise channels are trivially matched. Thus, in this work, we assume $n\geq 2$.

\begin{table*}[t!]
\centering
\caption{Summary of the control-affine Schrödinger bridge problems (CASBPs) and the nature of their solutions}
{\renewcommand{\arraystretch}{2}%
\begin{tabular}{| c | c | c | c |}
\hline
\diagbox{CASBP type}{Solution type} & Solvable & Algorithm & Performance guarantee\\ \hline
Same input and noise channels & linearly & dynamic Sinkhorn recursion & global convergence\\ \hline
Different input and noise channels & nonlinearly & dynamic Sinkhorn recursion with memory & local stability\\ 
& & (proposed in this work, Sec. \ref{subsec:ProposedAlgorithm}) & (Sec. \ref{sec:ConvergenceAnalysis})\\
\hline
\end{tabular}
}
\label{Table:CASBPSummary}
\end{table*}

\subsubsection*{Related works} Several works \cite{wakolbinger1990schrodinger,blaquiere1992controllability,caluya2020finite,caluya2021reflected,teter2024weyl,11239424,teter2025markov} have studied the stochastic control-theoretic generalizations of the classical Schr\"{o}dinger bridge, including the control-affine \cite{caluya2021wasserstein,chen2021stochastic,teter2025hopf,11208710} and control-non-affine \cite{nodozi2023neural,nodozi2023physics} cases. While works such as \cite{caluya2021wasserstein,chen2021stochastic} proposed solving the CASBP via dynamic Sinkhorn recursions, they made the assumption $\bm{gg}^{\top}\propto\bm{\Sigma}$ \emph{a priori} -- motivated partly by computational convenience, and partly by that the input and noise channels indeed coincide in practical situations such as when the process noise enters through stochastic actuation or external forcing. 

More recent works \cite{teter2025hopf,11208710} explicitly pointed out that for CASBP with channel mismatch, when the Hopf-Cole transformation \cite{hopf1950partial,cole1951quasi} is applied to the first order conditions for optimality, a boundary coupled system of linear PDEs are obtained, and that system can be solved by standard Sinkhorn algorithm, only if $\lambda\bm{gg}^{\top}-\bm{\Sigma}=\bm{0}$ for some $\lambda>0$. 

When the channels do not match, the Hopf-Cole-transformed PDE system remains nonlinear, and it is unclear whether a generalized variant of the Sinkhorn algorithm can be designed to solve the same.

In the linear Gaussian (i.e., covariance steering) special case of \eqref{CASBP}, the work in \cite{chen2015optimal} considered different input and noise channels, and derived\footnote{In the notation of reference \cite{chen2015optimal}, the term $\lambda\bm{gg}^{\top}-\bm{\Sigma}$ appears there as $\bm{BB}^{\top}-\bm{B}_{1}\bm{B}_{1}^{\top}$, where $\bm{B},\bm{B}_{1}$ are the state-independent input and noise coefficient matrices, respectively.} a system of coupled nonlinear matrix ODE boundary value problem \cite[eq. (11)]{chen2015optimal}. That reference approximated the solution of that matrix boundary value problem via time discretization and semidefinite programming. In contrast, the focus of this work is to solve \emph{nonlinear non-Gaussian CASBPs with channel mismatch}.

\subsubsection*{Motivation} The technical motivation behind this work is to remedy the lack of a computational algorithm to solve the generic CASBP with channel mismatch. The channel mismatch is practical and also motivated by cases where the noise may represent modeling uncertainty in addition to disturbances or imperfections in actuation \cite[eq. (7.1)]{pan1999backstepping}, \cite[eq. (90)]{li2021stochastic}, or when the input is affected by a lag in actuation dynamics \cite[Sec. V, Example 2]{chen2015optimal}. In such cases, the proportionality relation $\bm{gg}^{\top}\propto\bm{\Sigma}$ is violated.

More broadly, solving stochastic optimal control problems of the form \eqref{CASBP}-\eqref{FeasibleSet} is motivated by shaping state distributions via feedback control \cite{brockett2012notes}, i.e., direct control of uncertainties. For example, the specification of the initial PDF $\rho_0$ can be interpreted as the initial uncertainty obtained from an estimator, and that of the terminal PDF $\rho_1$ as the desired statistical performance to achieve over the specified time horizon $[t_0,t_1]$.   

\subsubsection*{Connections with linearly solvable stochastic optimal control} The proportionality relation $\bm{gg}^{\top}\propto\bm{\Sigma}$, i.e., $\lambda\bm{gg}^{\top}-\bm{\Sigma}=\bm{0}$ for some $\lambda>0$, has appeared before in the stochastic optimal control literature in a different context. Specifically, it was pointed out in \cite{kappen2005linear,todorov2009efficient,horowitz2014linear} that when solving stochastic optimal control problems of the form \eqref{CASBP} with the exception that the endpoint PDF constraints are removed, and a terminal cost of the form $\mathbb{E}_{\rho^{\bm{u}}}\left[\phi(\bm{x}^{u}(t_1))\right]$ for some suitable $\phi(\cdot)$ is added in the objective \eqref{CASBPobj}, the associated Hamilton-Jacobi-Bellman PDE can be \emph{exactly transformed to a linear PDE} provided $\lambda\bm{gg}^{\top}-\bm{\Sigma}=\bm{0}$. Such stochastic optimal control problems are called \emph{linearly solvable} \cite{kappen2005linear,todorov2009efficient,horowitz2014linear} since by standard dynamic programming arguments \cite[Ch. 4]{fleming2006controlled}, the optimal control can then be recovered from the solution of this transformed linear PDE.

When the relation $\bm{gg}^{\top}\propto\bm{\Sigma}$ does not hold, then the problem is not linearly solvable in the sense that computing the optimal control requires solving the second-order \emph{nonlinear} Hamilton-Jacobi-Bellman PDE. 

We will explain in Sec. \ref{sec:SolutionOfCASBP} that in the CASBP too, the lack of the same proportionality relation brings nonlinearity in a different way.

\subsubsection*{Contributions} Our contributions are the following.
\begin{itemize}
\item To solve the system of nonlinear PDEs arising in generic CASBP with channel mismatch, we propose a new dynamic Sinkhorn algorithm (Sec. \ref{subsec:ProposedAlgorithm}) that uses--in each epoch--memory from the most recent backward pass to be able to perform a forward pass in time.

\item We provide a local stability guarantee (Sec. \ref{sec:ConvergenceAnalysis}) for the proposed Sinkhorn algorithm with memory.

\item We demonstrate the effectiveness of the proposed generalization of the Sinkhorn algorithm via numerical examples (Sec. \ref{sec:NumericalExample}).
\end{itemize}

At a conceptual level, this work clarifies that when the input and noise channels match, then the CASBP is \emph{linearly solvable} in the sense the problem reduces to solving a boundary coupled system of \emph{linear PDEs} in the so-called Schr\"{o}dinger factors, and \emph{standard dynamic Sinkhorn recursions} apply with global convergence guarantees. 

In contrast, when these channels do not match, then the problem reduces to solving a boundary coupled system of \emph{nonlinear PDEs} in the Schr\"{o}dinger factors, which is shown to be solvable by \emph{dynamic Sinkhorn recursions with memory}. However, in this more general case, we only establish local stability guarantees. In this sense, the CASBP then is \emph{locally nonlinearly solvable}. For the readers' convenience, this is summarized in Table \ref{Table:CASBPSummary}.    

\subsubsection*{Organization} Sec. \ref{sec:Prelim} gathers some technical preliminaries that are used later in this work. In Sec. \ref{sec:SolutionOfCASBP}, we explain the main ideas in the context of state-of-the-art results from existing literature, and then propose the new algorithm. Sec. \ref{sec:ConvergenceAnalysis} undertakes an analysis of the proposed Sinkhorn algorithm with memory, and establishes a stability guarantee. Numerical results in Sec. \ref{sec:NumericalExample} illustrate that the proposed algorithm works well in practice. Sec. \ref{sec:Conclusions} concludes the work.


\section{Preliminaries}\label{sec:Prelim}
We begin our exposition by collecting some concepts and results needed in the sequel.  

\subsection{Cones, Hilbert's Projective Metric, and Contraction}\label{subsec:ConePrelim}
For $D \subseteq \mathbb R^n$, we consider the Banach space $L^\infty(D)$ with norm
$\Vert \cdot \Vert_\infty := \mathrm{ess} \sup_{\bm{x}\in D} \vert \cdot \vert$. Throughout, we view all extrema and inequalities involving functions in $L^\infty(D)$ in the almost-everywhere sense, and to reduce clutter, we suppress the $\mathrm{ess}$ notation.

Let 
\begin{align}
    K := \{ f \in L^\infty(D): f\geq  0 \}
\label{defCone}    
\end{align}
be the closed solid cone of nonnegative functions in $L^\infty(D)$, and denote the interior of $K$ by $K_+$. That is,
\begin{align}
    K_+ := \{ f \in L^\infty(D): \inf_{\bm{x}\in D} f(\bm{x}) >  0\}.
\label{defConeInterior}    
\end{align}

Given any $u,v \in K_+$, \emph{Hilbert's projective metric} $d_{{\mathrm{H}}}(\cdot,\cdot)$ is defined as  
\begin{align}
    d_{{\mathrm{H}}}(u,v):=\log\left(\dfrac{\inf\{\theta >0 \mid u\preceq \theta v\}}{\sup\{\theta >0 \mid \theta v\preceq u\}}\right)
\label{HilbertMetric}    
\end{align}
where $\preceq$ denotes the conic inequality, i.e., $u \preceq v$ if and only if $v - u \in K$, see \cite{birkhoff1957extensions,bushell1973hilbert,georgiou2015positive}.
Equivalently, 
\begin{align}\label{def:H-alt}
    d_{{\mathrm{H}}}(u,v) =\log \sup_{\bm{x}} \frac{u(\bm{x})}{v(\bm{x})} - \log \inf_{\bm{x}} \frac{u(\bm{x})}{v(\bm{x})} \quad \forall u, v \in K_+.
\end{align}
The term \emph{projective} refers to invariance under positive scalings, i.e., $d_{\mathrm{H}}(u, v) =0$ whenever $u = cv$ for any $c>0$. 

\begin{definition}[$p$-Homogeneity]\label{def:homogeneity}
Consider $K,K_+$ as in \eqref{defCone}-\eqref{defConeInterior}. Let $\mathcal G: K_+ \to K_+$ be a map on the interior of $K$, i.e., a \emph{positive} map. If 
    \begin{align}
        \mathcal G(\alpha u) = \alpha^p \mathcal G(u) \quad \forall u \in K_+, \;\alpha>0,
    \end{align}
    we say $\mathcal G$ is homogeneous of degree $p$ (or $p$-homogeneous).
\end{definition}
\begin{definition}[Global Contraction on $K_+$]\label{def:contractionratio}
    Consider $K,K_+$ as in \eqref{defCone}-\eqref{defConeInterior}. The map $\mathcal G:K_+ \to K_+$ is called globally nonexpansive with respect to $d_{\mathrm{H}}(\cdot, \cdot)$ if there exists $0<\kappa\leq 1$ such that
    \begin{align}
        d_{\mathrm{H}}(\mathcal G(u),\mathcal G(v)) \leq \kappa \, d_{\mathrm{H}}(u,v) \quad\forall u,v \in K_+.
    \label{defGloballyNonexpansiveInHilbertMetric}    
    \end{align}
     The map $\mathcal G$ is said to be globally contractive if \eqref{defGloballyNonexpansiveInHilbertMetric} holds for $0<\kappa \lneq1 $. The smallest $\kappa$ for which the map is contractive is called the contraction ratio of $\mathcal G$, denoted as $\kappa_{\rm H}(\mathcal G)$.
\end{definition}

For a positive \emph{linear} map $\mathcal G$, a relevant quantity is the \emph{projective diameter}
\begin{align}
{\mathrm{ProjDiam}}(\mathcal G)  := \sup_{u,v \in K_+} d_{\rm H} (\mathcal G(u), \mathcal G(v)).
\label{defProjectiveDiameter}
\end{align}
An inequality \cite[p. 2380]{chen2016entropic}
that will be used later is 
\begin{align} \label{eq:proj-diam-ineq}
    {\mathrm{ProjDiam}}(\mathcal G) \leq 2 \sup_{u \in K_+} d_{\rm H} (\mathcal G(u), v) ,
\end{align}
for any $v \in K_+$.

\begin{theorem}  [Birkhoff-Bushell Theorem \cite{bushell1973hilbert,birkhoff1957extensions}]\label{thm:BBtheorem}
Consider $K, K_+$ as in \eqref{defCone}-\eqref{defConeInterior}, the mapping $\mathcal G: K_+ \to K_+$, and the contraction ratio as in Definition \ref{def:contractionratio}. If $\mathcal G$ is homogeneous of degree $p$, and is monotone increasing, i.e., 
     \begin{align*}
        u \preceq v  \quad \Longrightarrow  \quad \mathcal G(u) \preceq \mathcal G(v),
     \end{align*}
then \eqref{defGloballyNonexpansiveInHilbertMetric} holds with
     \begin{align}
         \kappa_{\rm H}(\mathcal G)  \leq p.
     \end{align}
In the special case $\mathcal G$ is linear,
\begin{align}
  \kappa_{\rm H}(\mathcal G) \leq \tanh\left( \frac{1}{4}   {\mathrm{ProjDiam}}(\mathcal G)\right).
\end{align}
That is, if $\mathcal G:K_+ \to K_+$ is linear, and ${\mathrm{ProjDiam}}(\mathcal G) < \infty$, then $\mathcal G$ is globally contractive. 
\end{theorem}

\subsection{Comparison Principle for Parabolic PDEs}\label{subsec:parabolicPDEprelim}

\begin{theorem}[Comparison Principle {\normalfont \cite[Theorem 9.1]{lieberman1996second}}]
\label{thm:comparisonprinciple}
For $D\subset\mathbb{R}^{n}$, consider a spatio-temporal domain $\Omega := [t_0,t_1] \times D$ and its parabolic boundary
\begin{align}
 \mathscr{P}\Omega := (\{t_0\}\times \overline{D}) \cup ((t_0,t_1)\times\partial D),
 \end{align}
 where $\overline{D}, \partial D$ denote the topological closure and the boundary of $D$, respectively.
    Consider the quasilinear parabolic PDE
    \begin{align}
   -\partial_t& u(t,\bm{x}) \!+\! \sum_{i,j=1}^{n}a_{ij}(t,\bm{x},u,\nabla_{\bm{x}} u)\partial_{x_{i} x_{j}}u(t,\bm{x}) \nonumber \\
    &\qquad+ b(t,\bm{x},u,\nabla_{\bm{x}}u) =0, \quad \forall(t,\bm{x}) \in \Omega.
    \label{eq:QuasilinearParabolicPDE}
    \end{align}
    
In \eqref{eq:QuasilinearParabolicPDE}, suppose that $a_{ij}$ are independent of $u$, and $b$ is decreasing in $u$. 
Then, for smooth solutions $u$ and $v$ of \eqref{eq:QuasilinearParabolicPDE},
    we have
    \begin{align}
     u\leq v \text{ on } \mathscr{P}\Omega \quad\Longrightarrow \quad u\leq v \text{ on } [t_0,t_1] \times \overline{D}.
     \end{align}   
\end{theorem}



\section{Solution of the CASBP and its Computation}\label{sec:SolutionOfCASBP}
In Sec. \ref{subsec:SolutionStructure}, we discuss the solution structure for the generic CASBP \eqref{CASBP}-\eqref{FeasibleSet}. In Sec. \ref{subsec:SinkhornForMatchedChannel}, we explain why the case $\bm{gg}^{\top}\propto\bm{\Sigma}$ is amenable to dynamic Sinkhorn recursion. Then in Sec. \ref{subsec:MemoryFromBackwardPass}, we motivate our new idea of dynamic Sinkhorn recursion with memory for the case $\bm{gg}^{\top}\not\propto\bm{\Sigma}$. The proposed algorithm is summarized in Sec. \ref{subsec:ProposedAlgorithm}. 


\subsection{Solution Structure for the CASBP}\label{subsec:SolutionStructure}
Let $\langle\cdot,\cdot\rangle$ denote the Frobenius inner product, and 
\begin{align*}\Delta_{\bm{\Sigma}}\rho:=\displaystyle\sum_{i,j=1}^{n}\partial_{x_{i }x_{j}}\left(\bm{\Sigma}_{ij}(t,\bm{x})\rho(t,\bm{x})\right)
\end{align*}
denote the weighted Laplacian. We use $\Delta$ to denote the standard Laplacian, i.e., $\Delta\equiv\Delta_{\bm{I}}$.

It is known \cite[Sec. III]{teter2025hopf} that the first-order conditions for optimality for the CASBP \eqref{CASBP}-\eqref{FeasibleSet} are
\begin{subequations}
\begin{align}
&\partial_{t}\rho^{\bm{u}}_{\mathrm{opt}} + \nabla_{\bm{x}}\cdot\big(\rho^{\bm{u}}_{\mathrm{opt}}\big(\bm{f}+\bm{g}\bm{g}^{\top}\nabla_{\bm{x}}S\big)\big) = \frac{1}{2}\Delta_{\bm{\Sigma}}\:\rho^{\bm{u}}_{\mathrm{opt}}, \nonumber
\\ 
&\hspace*{2.25in}\text{(primal PDE)} \label{PrimalPDE}
 \\ &\partial_{t} S +\langle\nabla_{\bm{x}}S,\bm{f}\rangle\! +\!\frac{1}{2}\langle\nabla_{\bm{x}}S,\bm{gg}^{\top}\nabla_{\bm{x}}S\rangle +\frac{1}{2}\langle\bm{\Sigma},\hess_{\bm{x}}S\rangle = q,
\nonumber \\ 
&\hspace*{2.4in}\text{(dual PDE)} \label{DualPDE}\\
&\rho^{\bm{u}}_{\mathrm{opt}}\left(t=t_0,\cdot\right)=\rho_0(\cdot), \quad \rho^{\bm{u}}_{\mathrm{opt}}\left(t=t_1,\cdot\right)=\rho_1(\cdot). 
\nonumber  \\
&\hspace*{1.35in}(\text{primal boundary conditions}) \label{PrimalBC}
\end{align}
\label{FOCO}    
\end{subequations}%
\noindent The coupled system of equations \eqref{FOCO} is in unknown primal-dual pair $\left(\rho^{\bm{u}}_{\mathrm{opt}},S\right)$, wherein $\rho^{\bm{u}}_{\mathrm{opt}}$ is the optimally controlled joint state PDF, and $S\in\mathscr{C}^{1,2}\left([t_0,t_1],\mathbb{R}^{n};\mathbb{R}\right)$ is a Lagrange multiplier function\footnote{Here $\mathscr{C}^{1,2}\left([t_0,t_1],\mathbb{R}^{n};\mathbb{R}\right)$ denotes the class of scalar-valued functions $S(t,\bm{x})$ which are once continuously differentiable with respect to time $t\in[t_0,t_1]$ and twice continuously differentiable with respect to state $\bm{x}\in\mathbb{R}^{n}$.}. The optimal control is
\begin{align}
\bm{u}_{\mathrm{opt}} = \bm{g}^{\top}\nabla_{\bm{x}}S.
\label{OptimalControl}    
\end{align}

A direct numerical solution of \eqref{FOCO} is challenging since the nonlinear PDEs \eqref{PrimalPDE}-\eqref{DualPDE} are coupled in primal-dual pair $\left(\rho^{\bm{u}}_{\mathrm{opt}},S\right)$, but the boundary conditions \eqref{PrimalBC} are in terms of $\rho^{\bm{u}}_{\mathrm{opt}}$ only. 

To recast \eqref{FOCO} in a more amenable form, we fix a parameter $\lambda>0$, and consider the Hopf-Cole transform \cite{hopf1950partial,cole1951quasi}
$$\left(\rho^{\bm{u}}_{\mathrm{opt}},S\right)\mapsto\left(\widehat{\varphi},\varphi\right) := \left(\rho^{\bm{u}}_{\mathrm{opt}}\exp\left(-S/\lambda\right),\exp\left(S/\lambda\right)\right).$$
Applied to \eqref{FOCO}, this results in \cite[Theorem 2]{teter2025hopf} the transformed system:
\begin{subequations}
\begin{align}
&\!\!\partial_{t}\widehat{\varphi} + \nabla_{\bm{x}}\cdot\left(\widehat{\varphi}\left(\bm{f}+\bm{f}_{\varphi}\right)\right) - \dfrac{1}{2}\Delta_{\bm{\Sigma}}\:\widehat{\varphi} + \left(\dfrac{q}{\lambda}+q_{\varphi}\right)\widehat{\varphi} = 0, \label{varphihatPDE}\\ 
&\!\!\partial_{t}\varphi + \langle\nabla_{\bm{x}}\varphi,\bm{f}+\bm{f}_{\varphi}\rangle + \dfrac{1}{2}\langle\bm{\Sigma},\hess_{\bm{x}}\varphi\rangle -  \left(\dfrac{q}{\lambda}+q_{\varphi}\right)\varphi = 0, \label{varphiPDE}\\
&\widehat{\varphi}\left(t_0,\cdot\right)\varphi\left(t_0,\cdot\right)=\rho_0(\cdot), \quad \widehat{\varphi}\left(t_1,\cdot\right)\varphi\left(t_1,\cdot\right)=\rho_1(\cdot), \label{BilinearBC}
\end{align}
\label{BoundaryCoupledForwardBackward}   
\end{subequations}
where
\begin{align}
\bm{f}_{\varphi} &:= \left(\lambda\bm{gg}^{\top}-\bm{\Sigma}\right)\nabla_{\bm{x}}\log\varphi, \label{def_fphi}\\
q_{\varphi} &:= \dfrac{1}{2}\left(\nabla_{\bm{x}}\log\varphi\right)^{\top}\left(\lambda\bm{gg}^{\top}-\bm{\Sigma}\right)\nabla_{\bm{x}}\log\varphi. \label{def_qphi}
\end{align}
The minimizing pair $\left(\rho^{\bm{u}}_{\mathrm{opt}},\bm{u}_{\mathrm{opt}}\right)$ for the CASBP \eqref{CASBP}-\eqref{FeasibleSet} is then recovered from
\begin{align}
\rho^{\bm{u}}_{\mathrm{opt}}(t,\cdot) &= \widehat{\varphi}(t,\cdot)\varphi(t,\cdot),\label{OptimalPDFfromFactors}\\
\bm{u}_{\mathrm{opt}}(t,\cdot) &= \lambda\:\bm{g}^{\top}\nabla_{(\cdot)}\log\varphi(t,\cdot),\label{OptimalControlfromFactors}
\end{align}
for all $t\in[t_0,t_1]$. 

As standard, the pair $\left(\widehat{\varphi},\varphi\right)$ is referred to as the \emph{Schr\"{o}dinger factors} since by \eqref{OptimalPDFfromFactors}, they provide a factorization of $\rho^{\bm{u}}_{\mathrm{opt}}$, the optimally controlled joint state PDF.

\begin{remark}[Control weight matrix]\label{remark:ControlWeightMatrix}
It is easy to see that replacing $\frac{1}{2}\|\bm{u}\|_2^2$ in \eqref{CASBPobj} by $\frac{1}{2}\bm{u}^{\top}\bm{Ru}$ with $\bm{R}\succ\bm{0}$, generalizes the channel mismatch condition \eqref{DifferentChannel} as $\nexists\lambda>0$ such that $\lambda\bm{g}\bm{R}^{-1}\bm{g}^{\top}-\bm{\Sigma}=\bm{0}$. Then the weight matrix in both \eqref{def_fphi} and \eqref{def_qphi} becomes $(\lambda\bm{g}\bm{R}^{-1}\bm{g}^{\top}-\bm{\Sigma})$, and the optimal control in \eqref{OptimalControlfromFactors} becomes $\bm{u}_{\mathrm{opt}}(t,\cdot)=\lambda\bm{R}^{-1}\bm{g}^{\top}\nabla_{(\cdot)}\log\varphi(t,\cdot)$. 
\end{remark}

\begin{remark}[Value of $\lambda>0$]\label{remark:ValueOflambda}
When the channels match, then $\bm{gg}^{\top}\propto\bm{\Sigma}$, and the numerical value of $\lambda>0$, by definition, is the proportionality constant in the relation $\lambda\bm{gg}^{\top}-\bm{\Sigma}=\bm{0}$. When the channels do not match, then in all developments starting from Sec. \ref{subsec:MemoryFromBackwardPass}, we can set $\lambda=1$ without loss of generality. 
\end{remark}


\subsection{The Case \texorpdfstring{$\bm{gg}^{\top}\propto\bm{\Sigma}$: Dynamic Sinkhorn Recursion}{Propto}}\label{subsec:SinkhornForMatchedChannel}
Being a system of nonlinear PDE boundary value problem (BVP), the transformed system \eqref{BoundaryCoupledForwardBackward} at first glance appears to be as difficult as the original system \eqref{FOCO}. However, a closer inspection reveals an interesting structure in \eqref{BoundaryCoupledForwardBackward}. 

When the input and noise channels match, i.e., $\lambda\bm{gg}^{\top}-\bm{\Sigma}=\bm{0}$ for some $\lambda>0$, then both \eqref{def_fphi} and \eqref{def_qphi} vanish, and the PDEs \eqref{varphihatPDE}-\eqref{varphiPDE} become \emph{decoupled and linear}. The only coupling for the BVP then comes from the bilinear boundary conditions \eqref{BilinearBC}. This paves the way for the \emph{dynamic Sinkhorn recursion}:
\begin{align}
\widehat{\varphi}_{0} ~\xrightarrow{\mathcal{F}}~ \widehat{\varphi}_{1} ~\xrightarrow{\rho_1/\widehat{\varphi}_{1}}~ \varphi_{1} ~\xrightarrow{\mathcal{B}}~ \varphi_{0} ~\xrightarrow{\rho_0/\varphi_{0}}~ \left(\widehat{\varphi}_{0}\right)_{\text{next}},
\label{DynamicSinkhornRecursion}   
\end{align}
where
$$\widehat{\varphi}_{i}(\cdot):=\widehat{\varphi}(t=t_i,\cdot), \quad \varphi_{i}(\cdot):=\varphi(t=t_i,\cdot) \quad\forall i\in\{0,1\}.$$
In \eqref{DynamicSinkhornRecursion}, the \emph{forward-in-time map} $\mathcal{F}$ solves a linear PDE initial value problem (IVP) with \eqref{varphihatPDE} from $t_0$ to $t_1$, with initial condition $\widehat{\varphi}_0$. The \emph{backward-in-time map} $\mathcal{B}$ solves a linear PDE IVP with \eqref{varphiPDE} from $t_1$ to $t_0$, with initial condition $\varphi_1$. In both cases, solution is made possible by $\lambda\bm{gg}^{\top}-\bm{\Sigma}=\bm{0}$. 

It is well-known \cite{chen2016entropic} that such dynamic Sinkhorn recursions are contractive with worst-case linear rate-of-convergence with respect to Hilbert's projective metric
$d_{{\mathrm{H}}}$ defined in \eqref{HilbertMetric}. Specifically, the recursion \eqref{DynamicSinkhornRecursion} is guaranteed to converge to a pair $(\widehat{\varphi}_{0},\varphi_{1})$ that is unique in a projectivized sense, thereby so is the pair $\left(\widehat{\varphi}(t,\cdot),\varphi(t,\cdot)\right)$ $\forall t\in[t_0,t_1]$. Here ``unique in a projectivized sense" means unique up to reciprocal scaling by any constant $c>0$, i.e., if $(\widehat{\varphi}_{0},\varphi_{1})$ is a solution for \eqref{DynamicSinkhornRecursion}, then so is $(c\widehat{\varphi}_{0},\varphi_{1}/c)$ for any $c>0$. Thus, the computed pair $\left(c\widehat{\varphi}(t,\cdot),\varphi(t,\cdot)/c\right)$ is unique up to the choice of $c>0$. 

The numerical value of the reciprocal scaling constant $c>0$ for a converged pair depends on the choice of the initial guess in recursion \eqref{DynamicSinkhornRecursion}. However, this choice does not affect the unique computation of the original variables of interest, viz.  $\rho^{\bm{u}}_{\mathrm{opt}}$ and $\bm{u}_{\mathrm{opt}}$, since their recovery from \eqref{OptimalPDFfromFactors}-\eqref{OptimalControlfromFactors} remains unaffected by reciprocal scaling.


\begin{algorithm}[t!]
\caption{Sinkhorn with memory 
}
\begin{algorithmic}[1]
    \Require Numerical tolerance $\varepsilon$, maximum number of iterations $\texttt{maxiter}$, positive initial guesses $\varphi_{\rm{init}}, \widehat{\varphi}_{\rm{init}}$
    \State $\left(\varphi_1,\widehat{\varphi}_{0}\right)\gets\left(\varphi_{\rm{init}},\widehat{\varphi}_{\rm{init}}\right)$ 
    \State $\texttt{idx} \gets 1$ \Comment{ Initialize recursion index}
    \While{($\texttt{err}>\varepsilon$) \textbf{and} $({\texttt{idx}} <  {\texttt{maxiter}})$}
        \State $\varphi_1^{\rm{old}}\gets\varphi_1$
        \State $\widehat\varphi_0^{\rm{old}}\gets\widehat\varphi_0$
        \State $\varphi_0  \gets \mathcal{B}(\varphi_1)$ \hspace*{-1in} \Comment{Solve the backward PDE IVP \eqref{varphiPDE}}
        \State $\varphi \gets (\varphi_{t})_{t\in[0,1]}$ \Comment{Store history}
        \State $\widehat{\varphi}_0 \gets \rho_0 / \varphi_0$
        \State $\widehat{\varphi}_1 \gets \mathcal{F}_{\varphi}(\widehat{\varphi}_0)$ \Comment{Solve the forward PDE IVP \eqref{varphihatPDE}}
        \State ${\varphi}_1 \gets \rho_1 / \widehat{\varphi}_1$
        \State $\texttt{err}\gets\max\{d_{\mathrm{H}}(\widehat\varphi_0,\widehat\varphi_0^{\rm{old}}),d_{\mathrm{H}}(\varphi_1,\varphi_1^{\rm{old}})\}$
        \State $\texttt{idx}\gets \texttt{idx}+1$    
    \EndWhile
    \State $\rho^{\bm{u}}_{\mathrm{opt}}\gets\widehat{\varphi}\cdot\varphi$
    \State $\bm{u}_{\mathrm{opt}}\gets\lambda {\bm{g}}^\top\nabla\log\varphi$
\end{algorithmic}
{\bf{Result:}} Optimal solution $\rho^{\bm{u}}_{\mathrm{opt}}$, $\bm{u}_{\mathrm{opt}}$
\label{alg:memhorn}
\end{algorithm}


\subsection{The Case \texorpdfstring{$\bm{gg}^{\top}\not\propto\bm{\Sigma}$}{DifferentChannel}: Dynamic Sinkhorn Recursion with Memory From the Backward Pass}\label{subsec:MemoryFromBackwardPass}

When $\lambda\bm{gg}^{\top}-\bm{\Sigma}\neq\bm{0}$, the terms \eqref{def_fphi} and \eqref{def_qphi} are nonzero in general, and the recursion \eqref{DynamicSinkhornRecursion} does not apply because then \eqref{varphihatPDE}-\eqref{varphiPDE} are neither linear nor decoupled. 

To make algorithmic headway in this more general case, we notice that the BVP \eqref{BoundaryCoupledForwardBackward} still has some structure. In particular, the PDE \eqref{varphiPDE} is nonlinear yet decoupled from \eqref{varphihatPDE}, i.e., the coupling among \eqref{varphihatPDE}-\eqref{varphiPDE} is one way via the additional terms $\bm{f}_{\varphi},q_{\varphi}$. This motivates our idea of designing a variant of Sinkhorn recursion where starting with an initial guess $\varphi_1$, one could backward integrate \eqref{varphiPDE}, i.e., evaluate the map $\mathcal{B}$ in \eqref{DynamicSinkhornRecursion}, and store the intermediates $\left(\varphi_t\right)_{t\in[t_0,t_1]}$ in a buffer/temporary memory to be accessed for evaluating the map $\mathcal{F}$ in \eqref{DynamicSinkhornRecursion} that follows. Here and in the sequel, we follow the notational convention that $\varphi_t (\cdot):= \varphi(t,\cdot)$ $\forall t\in[t_0,t_1]$.

In other words, for solving the IVP with \eqref{varphihatPDE}, the most recent $\left(\varphi_t\right)_{t\in[t_0,t_1]}$ from the backward pass could be substituted in \eqref{varphihatPDE}. Then the resulting \emph{linear} reaction-advection-diffusion PDE solution can be marched forward in time. This changes the standard Sinkhorn recursion \eqref{DynamicSinkhornRecursion} to the new variant:
\begin{align}
\includegraphics[width=0.89\linewidth, valign=c]{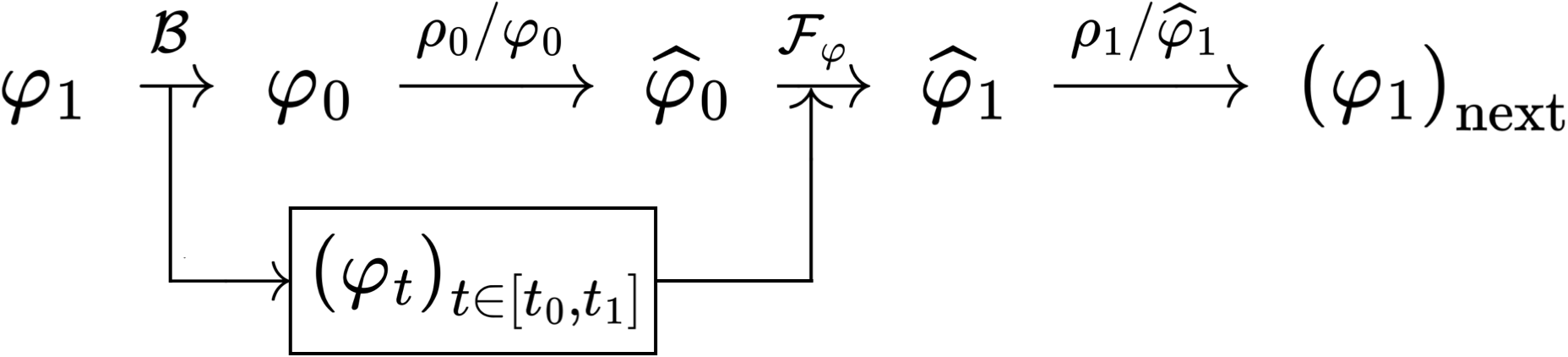}
\label{Memhorn}
\end{align} 
where the box depicts the buffer/temporary memory. In \eqref{Memhorn}, the subscript of the forward-in-time map $\mathcal{F}_{\varphi}$ emphasizes its dependence on the most recent history of the backward pass.

We refer to the recursion \eqref{Memhorn} as ``Sinkhorn with memory''. After each epoch, the buffer is flushed, and the $\left(\varphi_t\right)_{t\in[t_0,t_1]}$ is overwritten with the updated evaluation of $\mathcal{B}$. In the following, we concretize this idea into an algorithm.

\subsection{Proposed Algorithm}\label{subsec:ProposedAlgorithm}
Building on the ideas from the previous subsection, we propose Algorithm \ref{alg:memhorn} (Sinkhorn with memory) to solve the generic CASBP \eqref{CASBP}-\eqref{FeasibleSet} with input and noise channel mismatch.

Given the CASBP data $q,\bm{f},\bm{g},\bm{\sigma},\rho_0,\rho_1$, as in standard Sinkhorn recursion, Algorithm \ref{alg:memhorn} takes three additional inputs: numerical tolerance $\varepsilon$, maximum number of iterations $\texttt{maxiter}$, and initial guess $\varphi_{\rm{init}}$ (everywhere positive).

\begin{remark}[Symmetry-breaking]\label{remark:Asymmetry}
Unlike the standard dynamic Sinkhorn recursion \eqref{DynamicSinkhornRecursion}, there is an asymmetry in the proposed variant \eqref{Memhorn} in the sense that the proposed recursion must start from the initial guess $\varphi_1$ (the $\varphi_{\rm{init}}$ in Algorithm \ref{alg:memhorn}). In contrast, the recursion \eqref{DynamicSinkhornRecursion} can either start from the initial guess $\widehat{\varphi}_0$ or from $\varphi_1$. This symmetry-breaking originates from that \eqref{varphiPDE} does not depend on $\widehat{\varphi}(t,\cdot)$ while \eqref{varphihatPDE} depends on both $\widehat{\varphi}(t,\cdot),\varphi(t,\cdot)$. When the channels match, \eqref{varphihatPDE}-\eqref{varphiPDE} become equation-level decoupled, i.e., then they are only coupled through the boundary condition \eqref{BilinearBC}, thereby allowing the initial guess to be either $\widehat{\varphi}_0$ or $\varphi_1$.
\end{remark}

\noindent In the next Section, we focus on the analysis of Algorithm \ref{alg:memhorn}.


\section{Analysis of Dynamic Sinkhorn Recursion with Memory from the Backward Pass}\label{sec:ConvergenceAnalysis}

Given an initial guess, say $\varphi_{\rm init} =\varphi_1^{(1)}$, Algorithm \ref{alg:memhorn} is functionally a recursion of the form
\begin{align}
\varphi_1^{(j+1)} = \mathcal{T}_j(\varphi_1^{(j)})\:, \quad j=1,\hdots, \texttt{maxiter},
\label{OverallRecursionInVarphi1}    
\end{align} 
where each mapping $\mathcal{T}_j:K_{+}\to K_{+}$, with $K_+$ as in \eqref{defConeInterior}. 

Unlike the standard dynamic Sinkhorn recursion \eqref{DynamicSinkhornRecursion}, the map $\mathcal T_j$ is \emph{not fixed between iterations}, which is why we adjoin a subscript $j$ indicating the iteration index. Indeed, $\mathcal F_\varphi$ in \eqref{Memhorn} is dependent on the most recent trajectory $\varphi:=(\varphi_t)_{t\in [0,1]}$, which changes from an iteration to the next. 

For convenience, we introduce the following notations for the spatial differential operators corresponding to the backward and forward PDEs \eqref{varphiPDE} and \eqref{varphihatPDE}, respectively:
\begin{align}
\!\!\! \overleftarrow{\mathcal{L}}[\varphi] &:= -\langle\nabla_{\bm{x}}\varphi,\bm{f}+\bm{f}_{\varphi}\rangle - \dfrac{1}{2}\langle\bm{\Sigma},\hess_{\bm{x}}\varphi\rangle + \left(\dfrac{q}{\lambda}+q_{\varphi}\right)\varphi,
    \label{eq:L_backward} \\
\!\!\! \overrightarrow{\mathcal{L}}_\varphi[\widehat{\varphi}] &:= -\nabla_{\bm{x}}\cdot\left(\widehat{\varphi}\left(\bm{f} +\bm{f}_{\varphi}\right)\right) + \dfrac{1}{2}\Delta_{\bm{\Sigma}}\:\widehat{\varphi} - \left(\dfrac{q}{\lambda}+q_{\varphi}\right)\widehat{\varphi},
    \label{eq:L_forward}
\end{align}
where $ \overrightarrow{\mathcal{L}}_\varphi[\widehat{\varphi}]$ is meant as the operator for a fixed $\varphi$.

We also define the division operators 
\begin{align}
\mathcal{D}_i(\cdot) := \rho_i /(\cdot),  \quad i\in\{0,1\}.
\end{align}
For a fixed trajectory $\varphi$, we define a composite map $\mathcal T_\varphi$ as
\begin{align}
    \mathcal T_\varphi (\varphi_1) &:= \mathcal D_1(\mathcal F_\varphi(\mathcal D_0(\mathcal B(\varphi_1)))).
\end{align}  
Thus, at a fixed iteration index $j\in\{1,\hdots, \texttt{maxiter}\}$, 
\begin{align*}
\mathcal T_j = \mathcal T_{\varphi^{(j)}}.
\end{align*}

Standard dynamic Sinkhorn recursions, such as \eqref{DynamicSinkhornRecursion},  are proven \cite{chen2016entropic} to be globally contractive (Definition \ref{def:contractionratio})
with respect to Hilbert's projective metric $d_{\mathrm{H}}(\cdot, \cdot)$. Global contractivity leads to global convergence of iterations to a unique fixed point as highlighted earlier in Sec.~\ref {subsec:SinkhornForMatchedChannel}. It is then natural to ask what performance guarantee can be established for Algorithm \ref{alg:memhorn}. To address this question, we investigate the nature of the constituent maps of $\mathcal{T}_j$ to analyze the proposed algorithm with respect to $d_{\mathrm{H}}(\cdot, \cdot)$.

For the purpose of analyzing Algorithm \ref{alg:memhorn}, we will consider a compact spatial domain $D\subset\mathbb{R}^{n}$. The roadmap for our analysis is as follows. In Sec. \ref{subsec:global}, we discuss non-expansiveness for the maps $\mathcal{D}_0,\mathcal{D}_1,\mathcal B$. In Sec. \ref{subsec:local}, we establish sufficient conditions to ensure the map $\mathcal{F}_{\varphi}$ is locally contractive. In Sec. \ref{subsec:OverallLocalStability}, we bring these results together to prove local stability for Algorithm \ref{alg:memhorn} (Theorem \ref{thm:main}).

\subsection{Nonexpansiveness of $\mathcal{D}_0,\mathcal{D}_1,\mathcal B$} \label{subsec:global}
Using Definition~\ref{def:H-alt}, we note that for $i \in \{0,1\}$ and for all $u,v \in K_+$,
\begin{align*}
    d_{\rm H} (\mathcal{D}_i(u), \mathcal D_i(v)) &=  \!\log \sup_{\bm{x}} \frac{\rho_i(\bm{x})/u(\bm{x})}{\rho_i(\bm{x})/ v(\bm{x})} \!- \!\log \inf_{\bm{x}} \frac{\rho_i(\bm{x})/u(\bm{x})}{\rho_i(\bm{x})/ v(\bm{x})} \\
    &= d_{\rm H} (u, v).
\end{align*}
Hence, the division maps $\mathcal D_0$ and $\mathcal D_1$ are in fact isometries with respect to $d_{\rm H}(\cdot, \cdot)$. 

The next result establishes that the map $\mathcal{B}$ is non-expansive under a sufficient condition. 
\begin{proposition}[Non-expansiveness of $\mathcal{B}$]\label{prop:Bnonexpansive}
  If $q_{\varphi} + q/\lambda \geq 0$, then the backward-in-time map $\mathcal B: K_+ \to K_+$ is globally nonexpansive with respect to the Hilbert metric. That is, there exists a constant $0 \leq \kappa_{\mathcal B} \leq 1$ such that
    \begin{align}
        d_{\rm H} (\mathcal B(u), \mathcal B(v)) \leq \kappa_{\mathcal B}  d_{\rm H} (u, v) \quad \forall u, v \in K_+.
    \end{align}   
\end{proposition}
\begin{proof}
Our proof strategy is to verify that the map $\mathcal B$ is $1$-homogeneous (Definition~\ref{def:homogeneity}) and monotone increasing. We then invoke the Birkhoff-Bushell theorem (Theorem~\ref{thm:BBtheorem} in Sec. \ref{subsec:ConePrelim}) that establishes nonexpansiveness of $\mathcal {B}$.


For $\alpha>0$, let
\begin{align*}
    \bm{f}_{\alpha\varphi} &:= \left(\lambda\bm{gg}^{\top}-\bm{\Sigma}\right)\nabla_{\bm{x}}\log(\alpha\varphi), \\
    q_{\alpha\varphi} &:= \dfrac{1}{2}\left(\nabla_{\bm{x}}\log (\alpha \varphi)\right)^{\top}\left(\lambda\bm{gg}^{\top}-\bm{\Sigma}\right)\nabla_{\bm{x}}\log(\alpha\varphi).
\end{align*}
Since $\nabla_{\bm{x}}\log(\alpha\varphi) = \nabla_{\bm{x}}\log \varphi$, direct computation gives
   \begin{align*}
       \bm{f}_{\alpha\varphi} =  \bm{f}_\varphi, \quad q_{\alpha\varphi} =  q_{\varphi}.
   \end{align*}
Thus, from \eqref{eq:L_backward},
    \begin{align*}
        &\overleftarrow{\mathcal{L}}[\alpha\varphi]\\
        &= -\alpha\langle\nabla_{\bm{x}}\varphi,\bm{f}+\bm{f}_{\varphi}\rangle - \dfrac{\alpha}{2}\langle\bm{\Sigma},\hess_{\bm{x}}\varphi\rangle + \left(\dfrac{q}{\lambda}+q_{\varphi}\right)(\alpha\varphi) \\
        &= \alpha\overleftarrow{\mathcal{L}}[\varphi].
    \end{align*}
    That is, the differential operator that corresponds to $\mathcal {B}$, $\overleftarrow{\mathcal{L}}[\alpha\varphi]$, is homogeneous of degree $1$. Therefore, if $(\varphi_t)_{t\in[t_0,t_1]}$ is the solution to the backward PDE \eqref{varphiPDE} subject to $\varphi_1$, then 
    $\alpha (\varphi_t)_{t \in [t_0,t_1]}$ is also the solution to the backward PDE \eqref{varphiPDE} subject to $\alpha \varphi_1$. In turn,
    \begin{align*}
        \mathcal B (\alpha \varphi_1) = \alpha \mathcal B(\varphi_1)
    \end{align*}
    as $\mathcal B(\varphi_1) = \varphi_0$, establishing that $\mathcal B$ is $1$-homogeneous.
    
To prove the monotonicity of $\mathcal{B}$, we make use of the comparison principle for quasilinear parabolic PDEs (Theorem~\ref{thm:comparisonprinciple}). To this end, we define the operator
\begin{align*}
    P[\eta] := -\partial_t \eta-\overleftarrow{\mathcal{L}}[\eta],
\end{align*}
for $\eta \in K_+$.
Using \eqref{eq:L_backward}, $P[\eta]$ can be expressed as
\begin{align} \label{eq:parabolic operator}
     P[\eta] &=  -\partial_t \eta+\sum_{i,j=1}^{n}a_{ij}(t,\bm{x},\eta,\nabla_{\bm{x}} \eta)\partial_{x_{i} x_{j}}\eta(t,\bm{x}) \nonumber \\
    &\qquad+ b(t,\bm{x},\eta,\nabla_{\bm{x}}\eta),
\end{align}
where
\begin{align*}
    a_{ij}(t,\bm{x},\eta,\nabla_{\bm{x}} \eta) &= \frac{1}{2}\Sigma_{i,j}, \\
    b(t,\bm{x},\eta,\nabla_{\bm{x}}\eta) &= \langle\nabla_{\bm{x}} \eta,\bm{f}+\bm{f}_{\eta}\rangle - \left(\dfrac{q}{\lambda}+q_{\eta}\right)\eta,
\end{align*}
with \begin{align*}
    \bm{f}_{\eta} &:= \left(\lambda\bm{gg}^{\top}-\bm{\Sigma}\right)\nabla_{\bm{x}}\log \eta,\\
    q_{\eta} &:= \dfrac{1}{2}\left(\nabla_{\bm{x}}\log \eta \right)^{\top}\left(\lambda\bm{gg}^{\top}-\bm{\Sigma}\right)\nabla_{\bm{x}}\log \eta.
\end{align*}

Note that the coefficients $a_{ij}$ are independent of $\eta$. Further,
\begin{align}
    \partial_\eta{b} &= \partial_\eta\bigg[\langle\nabla_{\bm{x}} \eta,\bm{f}+\bm{f}_{\eta} \rangle - \left(\dfrac{q}{\lambda}+q_{\eta}\right)\eta\bigg] \nonumber\\
    &= -\frac{1}{2}\frac{(\nabla_{\bm{x}} \eta)^\top\left(\lambda\bm{gg}^{\top}-\bm{\Sigma}\right)(\nabla_{\bm{x}} \eta)}{\eta^2} - \dfrac{q}{\lambda}\nonumber\\
    &= -q_{\eta}-\dfrac{q}{\lambda},\label{BNonexpansiveSufficient}
\end{align}
so $b$ is decreasing in $\eta$ if $q_{\eta}+ q/\lambda\geq 0$. In this case, the comparison principle in Theorem \ref{thm:comparisonprinciple} applies to \eqref{eq:parabolic operator}.


Accordingly, we take $\eta(t,\cdot)=\varphi_{t_1-t}$ and $\zeta(t,\cdot)=\varphi_{t_1-t}'$ to be the solutions of the time reversal of \eqref{varphiPDE}, governed by the parabolic operators defined similarly to that in  \eqref{eq:parabolic operator} and subject to initial conditions $\eta(t_0, \cdot) =\varphi_1(\cdot)$ and $\zeta(t_0, \cdot) = \varphi_1'(\cdot)$, respectively.
If $\varphi_1 \leq \varphi_1'$, then by the comparison principle,
\begin{align*}
    \varphi_{t_1-t} = \eta(t,\cdot) \leq \zeta(t,\cdot) = \varphi_{t_1-t}' \quad\forall t\in[t_0,t_1].
    \end{align*}
    As $\mathcal{B}(\varphi_{1})=\eta(t_1,\cdot)$ and $\mathcal{B}(\varphi_{1}')=\zeta(t_1,\cdot)$, it follows that $\mathcal B$ is monotone increasing.

Having shown that $\mathcal B$ is $1$-homogeneous and monotone increasing, by Theorem~\ref{thm:BBtheorem}), we conclude that $\mathcal {B}$ is nonexpansive with respect to $d_{\mathrm{H}}(\cdot, \cdot)$.
\end{proof}

The following corollary offers a simpler sufficient condition in terms of the problem data.

\begin{corollary}[Positive semi-definiteness of $\lambda\bm{gg}^{\top}-\bm{\Sigma}$]\label{corr:NegDefWeightMatrix}
If $\lambda\bm{gg}^{\top}-\bm{\Sigma}\succeq\bm{0}$, then $q_{\varphi}+ q/\lambda \geq 0$, and Proposition \ref{prop:Bnonexpansive} applies.
\end{corollary}
\begin{proof}
Since the state cost $q\geq 0$, the expression \eqref{BNonexpansiveSufficient} is $\leq 0$ whenever $\lambda\bm{gg}^{\top}-\bm{\Sigma}\succeq\bm{0}$. In the absence of state cost ($q=0$), the conditions $q_{\varphi}+ q/\lambda\geq 0$ and $\lambda\bm{gg}^{\top}-\bm{\Sigma}\succeq\bm{0}$ are equivalent.  
\end{proof}

\subsection{The Forward Map $\mathcal{F}_\varphi$}\label{subsec:local}
Here, we derive sufficient conditions under which the forward mapping $\mathcal F_\varphi$ is locally contractive with respect to $d_{\rm H}(\cdot, \cdot)$ up to an additive residual. 

In contrast to the backward-in-time map $\mathcal{B}$, the map $\mathcal{F}_\varphi$ intrinsically depends on the trajectory $\varphi$ generated from the current iteration. Therefore, when we compare, say $\mathcal F_\varphi(\cdot)$ and $\mathcal F_{\varphi'}(\cdot)$, we need to account for the change in inputs of these maps, together with the change in the trajectories $\varphi, \varphi'$ that underlie their construction. This is formalized in the following triangle inequality:
\begin{align}
    d_{\mathrm{H}}(\mathcal{F}_{\varphi}(u),&\mathcal{F}_{\varphi'}(v)) \leq  d_{\mathrm{H}}(\mathcal{F}_{\varphi}(u),\mathcal{F}_{\varphi}(v)) \nonumber \\
    & + d_{\mathrm{H}}(\mathcal{F}_{\varphi}(v),\mathcal{F}_{\varphi'}(v)) \quad \forall u, v \in K_+.
    \label{eq:hilbertFtriangle}
\end{align}
We provide two complementary statements (Propositions~\ref{prop:Fvarphi_contractive} and \ref{prop:bound_dH-Fs}), each addressing how the forward mapping changes when its two dependencies change separately. Based on them, we point out how the mapping $\mathcal F_\varphi$  behaves locally under changing of both dependencies in Proposition~\ref{prop:combined}.

\begin{proposition}[Contraction of $\mathcal{F}_{\varphi}$ with $\varphi$ fixed]\label{prop:Fvarphi_contractive}
    For $\varphi$ fixed, the forward-in-time map $\mathcal{F}_\varphi$ is globally contractive in the Hilbert metric. 
\end{proposition}
\begin{proof}
With $\varphi$ fixed, \eqref{varphihatPDE} is a linear second-order parabolic PDE with differential operator $\overrightarrow{\mathcal{L}}_{\varphi}[\widehat{\varphi}]$ given as in \eqref{eq:L_forward}.
Applying Lemma~1 in \cite{teter2025hopf}, we verify that the operator coefficients
\begin{align*}
    \nabla_{\bm{x}}\cdot\bm{\Sigma}-(\bm{f}&+\bm{f}_\varphi), \\
      \frac{1}{2}\langle{\rm{Hess}}_{\bm{x}},\bm{\Sigma}\rangle &- \nabla_{\bm{x}}\cdot(\bm{f}+\bm{f}_\varphi) - \left(\frac{q}{\lambda} + q_\varphi\right)
\end{align*}
are smooth and bounded ({{due to \ref{A1}-\ref{A3}}}).

Then, for fixed $\varphi$, \eqref{varphihatPDE} is a linear parabolic PDE with smooth and bounded coefficients, and thereby, it admits \cite{aronson1968non,aronson1971non} a classical fundamental solution $\Gamma_\varphi(t,\bm{x},s,\bm{y})$. As such, given an initial condition $\widehat{\varphi}_{0}$ supported on a compact $D\subset\mathbb{R}^{n}$,
\begin{equation}
\widehat{\varphi}_{t}(\bm{x)} = \int_D \Gamma_\varphi(t,\bm{x},t_0,\bm{y})\widehat{\varphi}_{0}(\bm{y})\differential\bm{y}. 
\label{eq:def_U} 
\end{equation}
In particular,
\begin{align}
    \mathcal{F}_{\varphi}(\widehat{\varphi}_{0})(\bm{x}) = \widehat{\varphi}_{1}(\bm{x}) = \int_D \Gamma_\varphi(t_1,\bm{x},t_0,\bm{y})\widehat{\varphi}_{0}(\bm{y})\differential\bm{y}.
\label{GreenFunctionIntegral}    
\end{align}

Theorem 7 in \cite{aronson1968non} establishes that there exist finite constants $C, \alpha_1, \alpha_2 >0$ such that
\begin{align}
  \!\!  C^{-1}  \psi_1(t_1-t_0&,\bm{x}-\bm{y})  \leq \nonumber \\
  &\Gamma_\varphi(t_1,\bm{x},t_0,\bm{y})\nonumber\\
  &\leq C \psi_2(t_1-t_0,\bm{x}-\bm{y}) \quad \forall\bm{x},\bm{y} \in D,
\end{align}
where $\psi_i(t,\bm{x})$ is the fundamental solution of the heat equation 
\begin{align*}
    \partial_t \psi_i=\alpha_i \Delta \psi_i \quad i\in\{1,2\}.
\end{align*}
Let 
\begin{align*}
    m &:=  \inf_{\bm{x},\bm{y} \in D} C^{-1}  \psi_1(t_1-t_0,\bm{x}-\bm{y}), \\
    M &:= \sup_{\bm{x},\bm{y} \in D} C \psi_2(t_1-t_0,\bm{x}-\bm{y}).
\end{align*}
Then,
\begin{align}
    m  \leq \Gamma_\varphi(t_1,\bm{x},t_0,\bm{y}) \leq M \quad \forall \bm{x},\bm{y} \in D.
    \label{eq:Gammabound}
\end{align}
Since $0<\psi_i<\infty$ $\forall i\in\{1,2\}$, we have that $0<m\leq M<\infty$, constituting uniform bounds on the fundamental solution $\Gamma_\varphi(t_1,\bm{x},t_0,\bm{y})$.

For any $u\in K_+$, by \eqref{GreenFunctionIntegral} and \eqref{eq:Gammabound},
      \begin{align}
   \!\!  m\int_D u(\bm{y})\differential\bm{y}\leq \mathcal{F}_\varphi(u)(\bm{x}) \leq M \int_D u(\bm{y})\differential\bm{y} \quad \forall \bm{x} \in D.
     \end{align}
Letting $\bm{1}$ to be the 1-valued function on $D$, we bound the projective diameter (see \eqref{eq:proj-diam-ineq}) of $\mathcal F_\varphi$ as
    \begin{align*}
        {\mathrm{ProjDiam}}(\mathcal F_\varphi) &\leq 2\sup_{u\in K_+} d_{\rm{H}}(\mathcal{F}_\varphi(u),\bm{1}) \\
        &= 2\sup_{u\in K_+} \log\left(\frac{\sup_{\bm{x}\in D}\mathcal{F}_\varphi(u)(\bm{x})}{\inf_{\bm{x}\in D}\mathcal{F}_\varphi(u)(\bm{x})}\right) \\
        &\leq 2\sup_{u\in K_+} \log\left(\frac{M\int_D u(\bm{y})\differential\bm{y}}{m\int_D u(\bm{y})\differential\bm{y}}\right) \\
        &= 2\log\left(\frac{M}{m}\right) < \infty.
    \end{align*}
Then, by the Birkhoff-Bushell theorem (Theorem~\ref {thm:BBtheorem}), we have
    $$ \kappa_{\rm H}(\mathcal F_\varphi) \leq \tanh\left(\frac{\log(M/m)}{2}\right)<1.$$
Therefore, $\mathcal F_\varphi$, for a fixed $\varphi$, is globally contractive. 
\end{proof}

Before we proceed to our next statement, we define the operator $\Gamma_{\varphi}^{t\rightarrow s} \star u$ for $t_0\leq t\leq s\leq t_1$ as
\begin{align}\label{def:U-2}
\left(\Gamma_{\varphi}^{t\rightarrow s} \star u\right)(\bm{x}) &:= \int_D \Gamma_\varphi(s,\bm{x},t,\bm{y}) u(\bm{y}) \differential\bm{y},
\end{align}
where $\Gamma_\varphi(t,\bm{x},t,\bm{y}) = \delta(\bm{x}-\bm{y})$, the Dirac delta. 

We remind the reader of the notational convention from Sec. \ref{sec:SolutionOfCASBP} that $\widehat{\varphi}_t, \varphi_t$ refer to a pair of snapshots at $t\in[t_0,t_1]$, while $\widehat{\varphi}, \varphi$ (without the subscript $t$) refer to a pair of trajectories.

\begin{proposition}[Bounding the Hilbert metric between $\mathcal{F}_{\varphi}(v),\mathcal{F}_{\varphi'}(v)$]\label{prop:bound_dH-Fs}
Let $\varphi^*$ be a solution of \eqref{varphiPDE} subject to some final condition $\varphi_1^* \in K_+$ and $V_{\varphi^*}$ be a neighborhood of $\varphi^*$. Assume the following for all $\varphi, \varphi' \in V_{\varphi^*}$.
   \begin{enumerate}[label=\textbf{B\arabic*}]
        \item  \label{B1} There exist constants $m_V, M_V$ where $0 < m_V \leq M_V < \infty$ such that
        \begin{align}
                m_V  \leq \Gamma_\varphi(t_1,\bm{x},t_0,\bm{y}) \leq M_V \quad \forall \bm{x},\bm{y} \in D.
        \end{align} 
        \item \label{B1'} There exists a constant $0\leq \ell < \infty$ such that for all $t_0 \leq t \leq s \leq t_1$,
        \begin{align}
           \Vert \Gamma_{\varphi}^{t\rightarrow s} \star u \Vert_\infty \leq \ell \Vert u \Vert_\infty  \quad \forall u\in L^\infty(D).
        \end{align}
        \item \label{B2} There exist constants $0\leq c_i<\infty$, $i \in \{1,2\}$ such that
        \begin{align}
        \sup_{t \in [t_0,t_1]} \|\nabla^i_{\bm{x}}\log\varphi'_t-\nabla^i_{\bm{x}}\log\varphi_t\|_\infty &\leq c_i, \label{eq:bounds_gradlogdiff}
        \end{align}
        where $\nabla_{\bm{x}}^2 := \langle\nabla_{\bm{x}},\nabla_{\bm{x}}\rangle$, the standard Laplacian. 
        \item  \label{B3} There exists a constant $0\leq c_3< \infty$ such that
        \begin{align}
       \sup_{t \in [t_0,t_1]} \|\nabla_{\bm{x}}\log\varphi_t\|_\infty &\leq c_3. \label{eq:bounds_gradlog}
        \end{align}
        \item \label{B4} For every $v \in K_+$, there exist constants $0<A_i = A_i(v)< \infty$, $i\in \{0,1\}$, such that for every $\varphi \in V_{\varphi^*}$, if $\widehat \varphi$ is the solution to \eqref{varphihatPDE} that corresponds to $\varphi$ and subject to $\widehat \varphi_0 =v$, then
        \begin{align}
       \sup_{t \in [t_0,t_1]} \|\nabla^i_{\bm{x}} \widehat{\varphi}_t\|_\infty &\leq A_i\|v\|_\infty. \label{eq:bounds_varphihat}
        \end{align}
    \end{enumerate}
Now let $Z := c_1\|\nabla_{\bm{x}}\cdot\left(\lambda\bm{gg}^{\top}-\bm{\Sigma}\right)\|_\infty + (c_2 + c_1 c_3)\|\lambda\bm{gg}^{\top}-\bm{\Sigma}\|_\infty$ and let 
\begin{align}
    \alpha_{\mathcal{F}}(v) := \frac{\ell(A_1 c_1\|\lambda\bm{gg}^{\top}-\bm{\Sigma}\|_\infty +A_0 Z)}{m_V\int_D v(\bm{y}) \differential\bm{y}} (t_1-t_0)\|v\|_\infty,
    \label{eq:def_alpha_F}
    \end{align}
    for all $v \in K_+$.
Then, for every $v \in K_+$ and $\varphi, \varphi' \in V_{\varphi^*}$, we have
    \begin{equation}
    d_{\mathrm{H}}(\mathcal{F}_{\varphi}(v),\mathcal{F}_{\varphi'}(v)) \leq\log\left(\frac{1+\alpha_\mathcal{F}(v)}{1-\alpha_\mathcal{F}(v)}\right),
    \label{eq:bound_dH-different-varphi}
    \end{equation}
    provided that $\alpha_\mathcal{F}(v)<1$.
\end{proposition}

\begin{proof}
    Let $\varphi, \varphi' \in V_{\varphi^*}$. Also let $\widehat \varphi, \widehat \varphi'$ be solutions of \eqref{varphihatPDE} that correspond to $\varphi$ and $\varphi'$, respectively, subject to the same initial condition    $\widehat{\varphi}_0=\widehat{\varphi}_0'=v$, $v \in K_+$.

Letting $z_t:=\widehat{\varphi}_t-\widehat{\varphi}_t'$, and accordingly the trajectories $z:=\widehat{\varphi}-\widehat{\varphi}'$, 
    observe that 
    \begin{align*}
        \partial_t z &= \partial_t\widehat{\varphi}-\partial_t\widehat{\varphi}' \\
         &= \overrightarrow{\mathcal{L}}_{\varphi}[\widehat{\varphi}] - \overrightarrow{\mathcal{L}}_{\varphi'}[\widehat{\varphi}'] \\
        &= \overrightarrow{\mathcal{L}}_{\varphi}[z] + \underbrace{\left(\overrightarrow{\mathcal{L}}_{\varphi}[\widehat{\varphi}']-\overrightarrow{\mathcal{L}}_{\varphi'}[\widehat{\varphi}']\right)}_{=:\mathcal{L}_\Delta[\widehat{\varphi}']}.
    \end{align*}
    The difference operator $\mathcal{L}_\Delta[\cdot]$ reads
    \begin{subequations}\label{eq:Ldelta}
    \begin{align}
\mathcal{L}_\Delta[\widehat{\varphi}'] = \langle\nabla_{\bm{x}}\widehat{\varphi}', b_\Delta\rangle + \widehat{\varphi}'c_\Delta 
     \end{align}
    where 
    \begin{align}
     b_\Delta &:= \bm{f}_{\varphi'} - \bm{f}_\varphi, \label{eq:b-delta} \\
     c_\Delta &:=  \nabla_{\bm{x}}\cdot(\bm{f}_{\varphi'} - \bm{f}_\varphi) + (q_{\varphi'}-q_\varphi). \label{eq:c-delta}
    \end{align}
   \end{subequations}
    
    By the nonautonomous variation-of-constants formula, also known as \emph{Duhamel's principle} \cite[Pg. 51]{evans2022partial}\cite{filali2003non}, it follows that 
    $$ z_s = \left(\Gamma_\varphi^{t_0\rightarrow s}\star z_{t_0}\right) + \int_{t_0}^s \left(\Gamma_\varphi^{t\rightarrow s}\star\mathcal{L}_\Delta[\widehat{\varphi}_t']\right)\differential t.$$
   In our case $z_{t_0}=0$, and so
    $$ z_{t_1} = \int_{t_0}^{t_1} \left(\Gamma_\varphi^{t\rightarrow t_1}\star\mathcal{L}_\Delta[\widehat{\varphi}_t']\right)\differential t. $$
     Further,
     \begin{align*}
    \|z_{t_1}\|_\infty &\leq \int_{t_0}^{t_1}\|\Gamma_\varphi^{t\rightarrow t_1}\star\mathcal{L}_\Delta[\widehat{\varphi}_t']\|_\infty  \differential t.
    \end{align*}   
    Due to \ref{B2}-\ref{B4}, $\mathcal{L}_\Delta[\widehat{\varphi}'_t] \in L^\infty(D)$ for all $t \in [t_0,t_1]$. Then by \ref{B1'}, we have
    \begin{align}
    \|z_{t_1}\|_\infty &\leq \ell  \int_{t_0}^{t_1} \Vert \mathcal{L}_\Delta[\widehat{\varphi}_t']\|_\infty   \differential t \nonumber\\
    &\leq   \ell (t_1-t_0) \sup_{t \in [t_0,t_1]} \Vert \mathcal{L}_\Delta[\widehat{\varphi}_t']\|_\infty. \label{eq:z-1-L-bound}
    \end{align}
    
Now observe that,
    \begin{align}
&\sup_{t \in [t_0,t_1]} \|b_\Delta(t)\|_\infty\nonumber\\
\stackrel{\eqref{eq:b-delta}}{=}&\sup_{t \in [t_0,t_1]}\|\bm{f}_{\varphi'}(t)-\bm{f}_{\varphi}(t)\|_\infty \nonumber \\
    \stackrel{\eqref{def_fphi}}{=}&  \sup_{t \in [t_0,t_1]} \|\left(\lambda\bm{gg}^{\top}-\bm{\Sigma}\right)(\nabla_{\bm{x}}\log\varphi'_t-\nabla_{\bm{x}}\log\varphi_t)\|_\infty \nonumber \\
    \stackrel{\ref{B2}}{\leq} & c_1 \|\lambda\bm{gg}^{\top}-\bm{\Sigma}\|_\infty. \label{eq:bound_bdelta}
    \end{align}
   We also have
    \begin{align}
    &\|c_\Delta(t)\|_\infty \stackrel{\eqref{eq:c-delta}}{=} \|\nabla_{\bm{x}}\cdot(\bm{f}_{\varphi'} - \bm{f}_\varphi)(t) + (q_{\varphi'}-q_\varphi)(t)\|_\infty  \nonumber \\
    &\leq \|\nabla_{\bm{x}}\cdot(\bm{f}_{\varphi'} - \bm{f}_\varphi)(t)\|_\infty +  \|(q_{\varphi'}-q_\varphi)(t)\|_\infty. \label{eq:c}
    \end{align}
    Note that
    \begin{align} \|&\nabla_{\bm{x}}\cdot(\bm{f}_{\varphi'} - \bm{f}_\varphi)(t)\|_\infty \stackrel{\eqref{def_fphi}}{\leq} \nonumber \\
   &\|(\nabla_{\bm{x}}\cdot\lambda\bm{gg}^{\top}-\bm{\Sigma}) \cdot (\nabla_{\bm{x}}\log\varphi_t'-\nabla_{\bm{x}}\log\varphi_t )\Vert_\infty \nonumber\\
   &+ \Vert\left(\lambda\bm{gg}^{\top}-\bm{\Sigma}\right)(\nabla^2_{\bm{x}}\log\varphi'_t-\nabla^2_{\bm{x}}\log\varphi_t)\|_\infty, \nonumber \\
   &\stackrel{\ref{B2}}{\leq}  c_1\|\nabla_{\bm{x}}\cdot\lambda\bm{gg}^{\top}-\bm{\Sigma}\|_\infty + c_2\|\lambda\bm{gg}^{\top}-\bm{\Sigma}\|_\infty, \label{eq:f}
   \end{align}
   and
   \begin{align}\label{eq:q}
        \|(q_{\varphi'}&-q_\varphi)(t)\|_\infty \stackrel{\eqref{def_qphi}}{\leq} \Big[ \frac{1}{2}  \|\nabla_{\bm{x}}\log\varphi'_t+\nabla_{\bm{x}}\log\varphi_t \|_\infty  \times \nonumber  \\
     &  \| \lambda\bm{gg}^{\top}-\bm{\Sigma} \|_\infty \|\nabla_{\bm{x}}\log\varphi'_t-\nabla_{\bm{x}}\log\varphi_t \|_\infty \Big], \nonumber \\
     & \stackrel{\ref{B3}}{\leq}  c_1 c_3\|\lambda\bm{gg}^{\top}-\bm{\Sigma}\|_\infty.
   \end{align}
    Therewith, from \eqref{eq:c}-\eqref{eq:q}, we have
    \begin{align}
        &\sup_{t \in [t_0,t_1]} \|c_\Delta(t)\|_\infty \leq \nonumber \\
        & c_1\|\nabla_{\bm{x}}\cdot\lambda\bm{gg}^{\top}-\bm{\Sigma}\|_\infty + (c_2 +c_1 c_3) \|\lambda\bm{gg}^{\top}-\bm{\Sigma}\|_\infty  =:Z. \label{eq:bound_cdelta}
    \end{align}
    
    From \eqref{eq:Ldelta}, we know that
    \begin{align*}
    \sup_{t\in [t_0,t_1]} \|\mathcal{L}_\Delta[\widehat{\varphi}_t']\|_\infty &\leq \sup_{t \in [t_0,t_1]} (\|\nabla_{\bm{x}} \widehat{\varphi}_t'\|_\infty\|b_\Delta(t)\|_\infty)\\
    &+ \sup_{t \in [t_0,t_1]} (\|\widehat{\varphi}_t'\|_\infty\|c_\Delta(t)\|_\infty).
    \end{align*}
   By using \eqref{eq:bound_bdelta} and \eqref{eq:bound_cdelta}, we obtain
    \begin{align}
       \!\!\!   \sup_{t\in [t_0,t_1]} \|\mathcal{L}_\Delta[\widehat{\varphi}_t']\|_\infty &\leq \sup_{t \in [t_0,t_1]} \|\nabla_{\bm{x}} \widehat{\varphi}_t'\|_\infty c_1 \|\lambda\bm{gg}^{\top}-\bm{\Sigma}\|_\infty  \nonumber \\
    &+ \sup_{t \in [t_0,t_1]} \|\widehat{\varphi}_t'\|_\infty Z \nonumber \\
    & \hspace{-20pt}\stackrel{\ref{B4}}{\leq} (A_1 c_1\|\lambda\bm{gg}^{\top}-\bm{\Sigma}\|_\infty + A_0 Z)\|v\|_\infty.
    \end{align}
      By substituting the previous inequality into \eqref{eq:z-1-L-bound}, we have
      \begin{align} \label{eq:z_1-bound}
 \|z_{t_1}\|_\infty \leq \ell (t_1-t_0)  (A_1 c_1\|\lambda\bm{gg}^{\top}-\bm{\Sigma}\|_\infty + A_0 Z)\|v\|_\infty.         
      \end{align}

       Now observe that
    \begin{align} \label{eq:z-inf}
        \frac{-\|z_{t_1}\|_\infty}{\inf_{\bm{x} \in D}\widehat{\varphi}_1'(\bm{x})} \leq \frac{z_{t_1}(\bm{x})}{\widehat{\varphi}_1'(\bm{x})} \leq \frac{\|z_{t_1}\|_\infty}{\inf_{\bm{x} \in D}\widehat{\varphi}_1'(\bm{x})},
    \end{align}
for all $\bm{x} \in D$ and further,
 \begin{align}\label{eq:inf-den}
     \inf_{\bm{x} \in D}\widehat{\varphi}_1'(\bm{x}) &= \inf_{\bm{x} \in D} \int \Gamma_{\varphi'} (t_1,\bm{x},t_0,\bm{y})v(\bm{y}) \differential \bm{y} \nonumber \\
     &\stackrel{\ref{B1}}{\geq}  m_V \int_D v(\bm{y}) d \bm{y}. 
 \end{align}
 By substituting \eqref{eq:z_1-bound} and \eqref{eq:inf-den} into \eqref{eq:z-inf}, we get
 \begin{align}\label{eq:z_1-ratio}
         -\alpha_{\mathcal F}(v) \leq \frac{z_{t_1}(\bm{x})}{\widehat{\varphi}_1'(\bm{x})} \leq \alpha_{\mathcal F}(v) \quad \forall \bm{x} \in D
 \end{align}
 where $\alpha_{\mathcal F}(v)$ is given in \eqref{eq:def_alpha_F}.
    
            Recall that $z_{t_1} = \widehat{\varphi}_1-\widehat{\varphi}_1'$. Consequently, 
    \begin{align} \label{eq:hatphi-z-ratio}
        \frac{\widehat{\varphi}_1}{\widehat{\varphi}_1'} = 1 + \frac{z_{t_1}}{\widehat{\varphi}_1'}.
    \end{align} 
   By \eqref{eq:z_1-ratio} and \eqref{eq:hatphi-z-ratio}, we obtain
    \begin{align*}
        1-   \alpha_{\mathcal F}(v) \leq  \frac{\widehat{\varphi}_1}{\widehat{\varphi}_1'} = \frac{\mathcal F_\varphi(v)}{\mathcal F_{\varphi'}(v)}  \leq 1 +  \alpha_{\mathcal F}(v).
    \end{align*}
   By definition of the Hilbert metric (Definition~\ref{def:H-alt}), 
   \begin{align}
       d_{\rm H}(\mathcal F_\varphi(v),\mathcal F_{\varphi'}(v)) \leq \log (1+ \alpha_{\mathcal F}) - \log(1- \alpha_{\mathcal F}),
   \end{align}
   resulting in the expression in \eqref{eq:bound_dH-different-varphi}.
\end{proof}

\begin{proposition}[Bounding the Hilbert metric between $\mathcal{F}_{\varphi}(u),\mathcal{F}_{\varphi'}(v)$]\label{prop:combined}
    Consider the neighborhood $V_{\varphi^*}$ and the assumptions in Proposition~\ref{prop:bound_dH-Fs}. Then for every $\varphi, \varphi' \in V_{\varphi^*}$ and $u, v \in K_+$,
 \begin{align}\label{eq:F-all}
      d_{\mathrm{H}}(\mathcal{F}_{\varphi}(u),\mathcal{F}_{\varphi'}(v)) &\leq \kappa_{\mathcal F}^* \,  d_{\mathrm{H}}(u,v) + \varepsilon(v),
 \end{align}
where 
\begin{align}
    \kappa_{\mathcal F}^*  &:=  \sup_{\varphi \in V_{\varphi^*}} \kappa_{\rm H}(\mathcal F_\varphi),   \label{eq:contraction-ration-V} \\
    \varepsilon(v) &:= \log \left(\frac{1+ \alpha_{\mathcal F}(v)}{1- \alpha_{\mathcal F}(v)} \right), \label{eq:delta}
\end{align}
provided that ${\alpha_{\mathcal F}(v) <1}$. In addition, $$\kappa_{\mathcal F}^*  \leq \tanh\left(\frac{\log(M_V/m_V)}{2}\right)<1.$$
\end{proposition}
\begin{proof}
    From Proposition~\ref{prop:Fvarphi_contractive}, we know that for a fixed $\varphi$, $\mathcal F_\varphi$ is contractive with contraction ratio $\kappa_{\mathrm{H}}(\mathcal F_\varphi)$ satisfying
\begin{align}\label{eq:contraction-varphi-const}
        \kappa_{\rm H}(\mathcal F_\varphi)   \leq \tanh\left(\frac{\log(M/m)}{2}\right)<1.
    \end{align}
    Similar arguments to those in the proof of Proposition \ref{prop:Fvarphi_contractive} together with \ref{B1} imply that
    \begin{align}
        \kappa_{\rm H}(\mathcal F_\varphi) \leq \tanh\left(\frac{\log(M_V/m_V)}{2}\right)<1
    \end{align}
    for all $\varphi \in V_{\varphi^*}$.
    Using the definition in \eqref{eq:contraction-ration-V}, then it holds that
    \begin{align}
         \kappa_{\mathcal F}^* \leq \tanh\left(\frac{\log(M_V/m_V)}{2}\right)<1.
    \end{align}
    Consequently, for all $\varphi \in V_{\varphi^*}$, it follows that 
     \begin{align}\label{eq:term-1-F}
      d_{\mathrm{H}}(\mathcal{F}_{\varphi}(u),\mathcal{F}_{\varphi}(v)) \leq \kappa_{\mathcal F}^* \, d_{\mathrm{H}}(u,v).
      \end{align}
      
    Also immediately from Proposition~\ref{prop:bound_dH-Fs}, we get that for all $\varphi,\varphi' \in V_{\varphi^*},v \in K_+$, 
    \begin{align}\label{eq:term-2-F}
        d_{\mathrm{H}}(\mathcal{F}_{\varphi}(v),\mathcal{F}_{\varphi'}(v)) \leq \varepsilon(v),
    \end{align}
    with $\varepsilon(v)$ given by \eqref{eq:delta}.  
   By the triangle inequality
   \begin{align*}
        d_{\mathrm{H}}(\mathcal{F}_{\varphi}(u),\mathcal{F}_{\varphi'}(v)) &\leq  d_{\mathrm{H}}(\mathcal{F}_{\varphi}(u),\mathcal{F}_{\varphi}(v))\\
        &+  d_{\mathrm{H}}(\mathcal{F}_{\varphi}(v),\mathcal{F}_{\varphi'}(v)),
   \end{align*}
  and the inequalities \eqref{eq:term-1-F} and \eqref{eq:term-2-F}, we arrive at \eqref{eq:F-all}.    
\end{proof}

\subsection{Local Stability of Algorithm \ref{alg:memhorn}}\label{subsec:OverallLocalStability}
Building on Propositions~\ref{prop:Bnonexpansive}-\ref{prop:combined}, our main result on the stability of Algorithm~\ref{alg:memhorn} is as follows.
\begin{theorem}\label{thm:main}
  Consider $\varphi_1^* \in K_+$ such that
   \begin{align}
       \varphi_1^* = \mathcal T_{\varphi^*}(\varphi_1^*),
       \end{align}
       i.e.,  $\varphi_1^*$ is a fixed point of the update map $\mathcal T_{\varphi^*}$ defined by the trajectory $\varphi^*$. Let
   \begin{align*}
       \widehat \varphi_0^* :=  \mathcal D_0(\mathcal B(\varphi_1^*)).
   \end{align*}
   For $r>0$, consider a Hilbert ball centered around $\varphi_1^*$, given by
      \begin{align}
        B_r(\varphi_1^*) := \{ \varphi_1 \in K_+ ~|~ d_{\rm H} (\varphi_1, \varphi_1^*) \leq r\},
    \end{align}
    such that whenever $\varphi_1 \in B_r(\varphi_1^*)$, its corresponding trajectory $\varphi$ belongs to $V_{\varphi^*}$.
    Under the assumptions of Propositions~\ref{prop:Bnonexpansive}-\ref{prop:combined}, if
    \begin{align} \label{eq:r-cond}
        r\geq \frac{\varepsilon(\widehat \varphi_0^*)}{1- \kappa_{\mathcal B} \kappa_{\mathcal F}^*},
    \end{align}
   and the initial guess of Algorithm~\ref{alg:memhorn}, $ \varphi_1^{(1)} =\varphi_{\rm init}$, belongs to $B_r(\varphi_1^*)$, then the updates $\varphi_1^{(j+1)}= \mathcal T_{j} (\varphi_1^{(j)})$ also belong to $B_r(\varphi_1^*)$ for all iteration index $j\in\mathbb{N}$.
    Additionally, 
    \begin{align}\label{eq:err-ball}
     \lim_{j \to \infty} d_{\rm H}(\varphi_1^{(j)}, \varphi_1^*) \leq \frac{\varepsilon(\widehat \varphi_0^*)}{1- \kappa_{\mathcal B} \kappa_{\mathcal F}^*}.
    \end{align}
\end{theorem}

\begin{proof}
    Let $\varphi_1 \in B_r(\varphi_1^*)$ and the corresponding trajectory $\varphi \in V_{\varphi^*}$. Also, let $\widehat \varphi_0 :=  \mathcal D_0(\mathcal B(\varphi_1))$. 
    Since $\mathcal D_0: K_+ \to K_+$ is an isometry with respect to $d_{\rm H}(\cdot, \cdot)$, we have
    \begin{align} \label{eq:iso-1}
         d_{\rm H} (\widehat \varphi_0,  \widehat \varphi_0^*) =  d_{\rm H} (\mathcal B(\varphi_1), \mathcal B(\varphi_1^*)).
    \end{align}
    Under the assumption of Proposition~\ref {prop:Bnonexpansive}, $\mathcal B$ is non-expansive, thus satisfying   \begin{align}\label{eq:contraction-B-again}
         d_{\rm H} (\mathcal B(\varphi_1), \mathcal B(\varphi_1^*)) \leq \kappa_{\mathcal B} d_{\rm H} (\varphi_1, \varphi_1^*), 
    \end{align} with $\kappa_{\mathcal B} \leq 1$. Then, by \eqref{eq:iso-1} and \eqref{eq:contraction-B-again}, we obtain
    \begin{align} \label{eq:dis-hatvarphi0}
        d_{\rm H} (\widehat \varphi_0,  \widehat \varphi_0^*) \leq \kappa_{\mathcal B} d_{\rm H} (\varphi_1, \varphi_1^*).
    \end{align}

    We also note that
  \begin{align} \label{eq:iso-2}
   d_{\rm H} (\mathcal T_\varphi(\varphi_1), \varphi_1^*) &= d_{\rm H} (\mathcal T_\varphi(\varphi_1), \mathcal T_{\varphi^*}(\varphi_1^*))  \nonumber \\
   &= d_{\rm H}(\mathcal F_\varphi(\widehat\varphi_0), \mathcal F_{\varphi^*}(\widehat\varphi_0^*)),
\end{align}
because $\mathcal D_1:K_+ \to K_+$ is an isometry with respect to $d_{\rm H}(\cdot, \cdot)$. 
Under the assumptions of Propositions~\ref{prop:Fvarphi_contractive}-\ref{prop:combined}, we have
\begin{align}\label{eq:contraction-F-again}
      d_{\rm H}(\mathcal F_\varphi(\widehat\varphi_0), \mathcal F_{\varphi^*}(\widehat\varphi_0^*)) \leq \kappa_{\mathcal F}^* \, d_{\rm H} (\widehat \varphi_0,  \widehat \varphi_0^*) + \varepsilon (\widehat \varphi_0^*).
\end{align}
Then, by \eqref{eq:iso-2} and \eqref{eq:contraction-F-again}, we obtain
\begin{align}\label{eq:dist-T}
     d_{\rm H} (\mathcal T_\varphi(\varphi_1), \varphi_1^*) \leq \kappa_{\mathcal F}^* \, d_{\rm H} (\widehat \varphi_0,  \widehat \varphi_0^*) + \varepsilon (\widehat \varphi_0^*).
\end{align}
By \eqref{eq:dis-hatvarphi0} and \eqref{eq:dist-T}, we get
\begin{align}\label{eq:dis-T-2}
     d_{\rm H} (\mathcal T_\varphi(\varphi_1), \varphi_1^*) \leq \kappa_{\mathcal B}\kappa_{\mathcal F}^* \, d_{\rm H} ( \varphi_1,   \varphi_1^*) + \varepsilon (\widehat \varphi_0^*).
\end{align}

The previous relation implies that
\begin{align*}
 d_{\rm H} ( \varphi_1^{(2)},   \varphi_1^*) &= d_{\rm H} (\mathcal T_{\varphi^{(1)}}(\varphi_1^{(1)}), \varphi_1^*) \\
 & \stackrel{\eqref{eq:dis-T-2}}{\leq} \kappa_{\mathcal B}\kappa_{\mathcal F}^* \, d_{\rm H} ( \varphi_1^{(1)},   \varphi_1^*) + \varepsilon (\widehat \varphi_0^*).
\end{align*}
If $\varphi_1^{(1)}$ belongs to $B_r(\varphi_1^*)$, then
\begin{align*}
     d_{\rm H} ( \varphi_1^{(2)},   \varphi_1^*) \leq r \kappa_{\mathcal B}\kappa_{\mathcal F}^*  + \varepsilon (\widehat \varphi_0^*).
\end{align*}
And as long as the condition \eqref{eq:r-cond} holds, $\varphi_1^{(2)} \in B_r(\varphi_1^*)$.

In a similar way, we can verify that
\begin{align*}
      d_{\rm H} (\varphi_1^{(j)},& \varphi_1^*) \leq  (\kappa_{\mathcal B} \kappa_{\mathcal F}^*)^{j-1}  d_{\rm H} ( \varphi_1^{(1)},   \varphi_1^*) +  \nonumber\\
  & \varepsilon (\widehat \varphi_0^*)(1+ \kappa_{\mathcal B} \kappa_{\mathcal F}^*+ \cdots +(\kappa_{\mathcal B} \kappa_{\mathcal F}^*)^{j-2}),
\end{align*}
and $\varphi_1^{(j)} \in B_r(\varphi^*)$ if the condition \eqref{eq:r-cond} holds. Further,
\begin{align}
    \lim_{j \to \infty}  d_{\rm H}(\varphi_1^{(j)},& \varphi_1^*) \leq  r\lim_{j \to \infty}  (\kappa_{\mathcal B} \kappa_{\mathcal F}^*)^{j-1} \nonumber\\
    &+ \varepsilon (\widehat \varphi_0^*) \lim_{j \to \infty}   \sum_{i=0}^{j-2} (\kappa_{\mathcal B} \kappa_{\mathcal F}^*)^i.\label{FinalInequality}
\end{align}
Since $\kappa_{\mathcal B} \kappa_{\mathcal F}^* < 1$, $\lim_{j \to \infty}  (\kappa_{\mathcal B} \kappa_{\mathcal F}^*)^j=0$, and the geometric sum for the second summand in \eqref{FinalInequality} yields \eqref{eq:err-ball}.
    \end{proof}

Theorem~\ref{thm:main} states that Algorithm~\ref{alg:memhorn} is locally stable around a fixed point $\varphi^*_1$. In that, if the initial guess lies within a Hilbert metric ball centered around $\varphi_1^*$, consecutive iterates remain in this ball, provided that its radius satisfies \eqref{eq:r-cond}. 

Further, this Theorem shows that Algorithm~\ref{alg:memhorn} converges to a Hilbert metric ball around the fixed point, with radius not exceeding $\varepsilon(\widehat \varphi_0^*)/(1- \kappa_{\mathcal B} \kappa_{\mathcal F}^*)$, thereby providing an upper bound on the asymptotic error for Algorithm~\ref{alg:memhorn}. This error is sufficiently small, precisely $\varepsilon(\widehat \varphi_0^*)$ is small, for small channel mismatch and spatial rate of change for such mismatch, quantified respectively by $\Vert \lambda \bm{g g}^\top - \bm{\Sigma}\Vert_\infty$ and $\Vert \nabla_{\bm x} \cdot \left(\lambda \bm{g g}^\top - \bm{\Sigma}\right)\Vert_\infty$.

\begin{remark}[Warm start for Algorithm \ref{alg:memhorn}]
    The above discussion suggests a warm start or homotopic approach for practical implementation of Algorithm \ref{alg:memhorn}. For instance, if Algorithm \ref{alg:memhorn} fails to converge for large channel mismatch, one may take the initial guess $\varphi_{\rm{init}}$ to be a prior output (i.e., converged $\varphi_1$) of the Algorithm, solved then with small channel mismatch. 
\end{remark}

\section{Numerical Example}\label{sec:NumericalExample}
To illustrate the proposed algorithm, we consider steering the state PDF for a cubic spring-mass-damper with process noise. Specifically, we consider an instance of \eqref{CASBP} with $m=1$ input, $n=2$ states, $p=2$ noises, time horizon $[t_0,t_1]=[0,1]$, quadratic state cost $q = \frac{1}{2}(\bm{x}^{\bm{u}})^{\top}\bm{Q}\bm{x}^{\bm{u}}$ with $\bm{Q}={\mathrm{diag}}(1,2)$, $\lambda=1$, and
\begin{align*}
&\!\bm{f}(t,x_1,x_2) \!= \!\!\begin{pmatrix}
x_2\\
-\dfrac{\partial}{\partial x_1}V(x_1) - \gamma x_2
\end{pmatrix}\!\!,\; V(\cdot)=\frac{1}{4}(\cdot)^{4}, \; \gamma=1,\\
&\bm{g}(t,x_1,x_2) = \begin{pmatrix}
0\\
1
\end{pmatrix}, \; \bm{\sigma}(t,x_1,x_2)=\bm{I}.
\end{align*}
This is an instance of CASBP with channel mismatch because here, $$\bm{gg}^{\top}=\begin{bmatrix}
0 & 0\\
0 & 1
\end{bmatrix}$$ and $\bm{\Sigma}=\bm{I}$ are not proportional.

\begin{figure*}[t!]
    \centering
\includegraphics[width=0.99\linewidth]{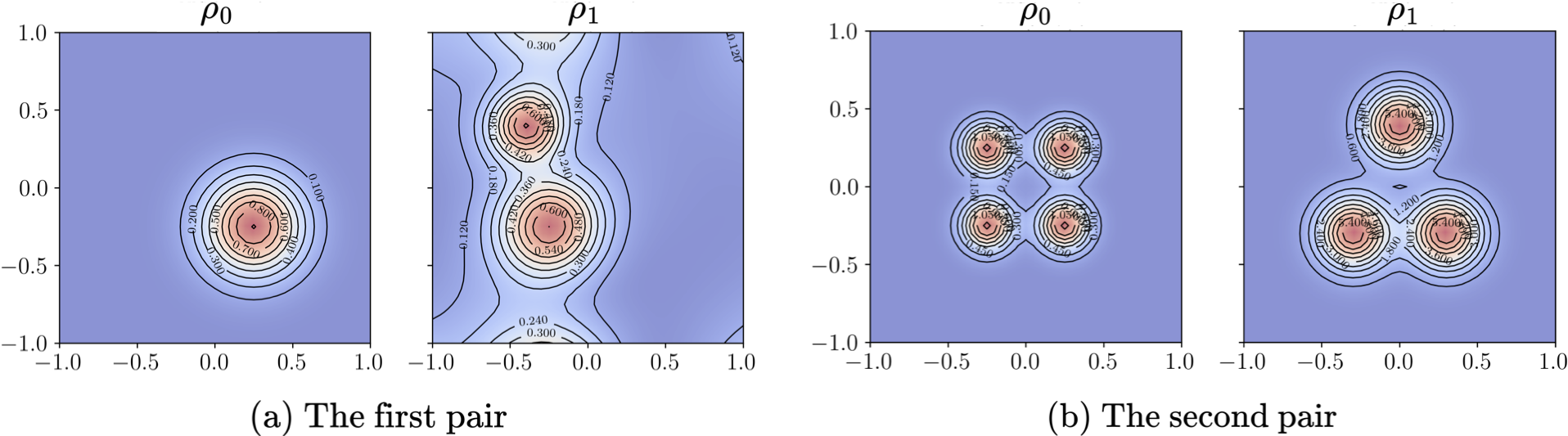}
\caption{The two pairs of endpoint PDFs used in the numerical example in Sec. \ref{sec:NumericalExample}.}
\label{fig:EndpointPDFs}
\end{figure*}

We solve the above instance of problem \eqref{CASBP} for two pairs of endpoint PDFs (Fig. \ref{fig:EndpointPDFs}): one pair being 
\begin{align*}
&\rho_{0}(\bm{x}) = \mathcal{N}\left(\begin{pmatrix}0.25\\
-0.25\end{pmatrix},\frac{1}{20}\bm{I}\right),\\
&\rho_{1}(\bm{x}) = \frac{1}{2}\mathcal{N} \left(\begin{pmatrix}-0.4\\
0.4\end{pmatrix}, \frac{1}{40}\bm{I}\right) +\!\frac{1}{2}\mathcal{N} \left(\begin{pmatrix}-0.25\\-0.25\end{pmatrix}, \frac{1}{20}\bm{I} \right),
\end{align*}
the other pair being
\begin{align*}
&\rho_{0}(\bm{x}) = \dfrac{1}{4}\sum_{i=1}^{4}\mathcal{N} \left(\bm{\mu}_{0i},\frac{1}{80}\bm{I}\right), \;\forall\bm{\mu}_{0i}\in\{-0.25,0.25\}^{2},\\
&\rho_{1}(\bm{x}) = \frac{1}{3}\mathcal{N} \left(\begin{pmatrix}0\\
0.4\end{pmatrix}, \frac{1}{40} \bm{I} \right) +\!\frac{1}{3}\mathcal{N} \left( \begin{pmatrix}-0.3\\
-0.3\end{pmatrix} , \frac{1}{40}\bm{I} \right)\\
&\qquad\qquad\qquad\qquad\qquad\qquad+\!\frac{1}{3}\mathcal{N} \left( \begin{pmatrix}0.3\\
-0.3\end{pmatrix} , \frac{1}{40}\bm{I} \right),
\end{align*}
all re-normalized over $[-1,1]^2$.

\begin{figure*}[t!]
    \centering
    \begin{subfigure}[b]{0.45\linewidth}
        \centering
        \includegraphics[width=\linewidth]{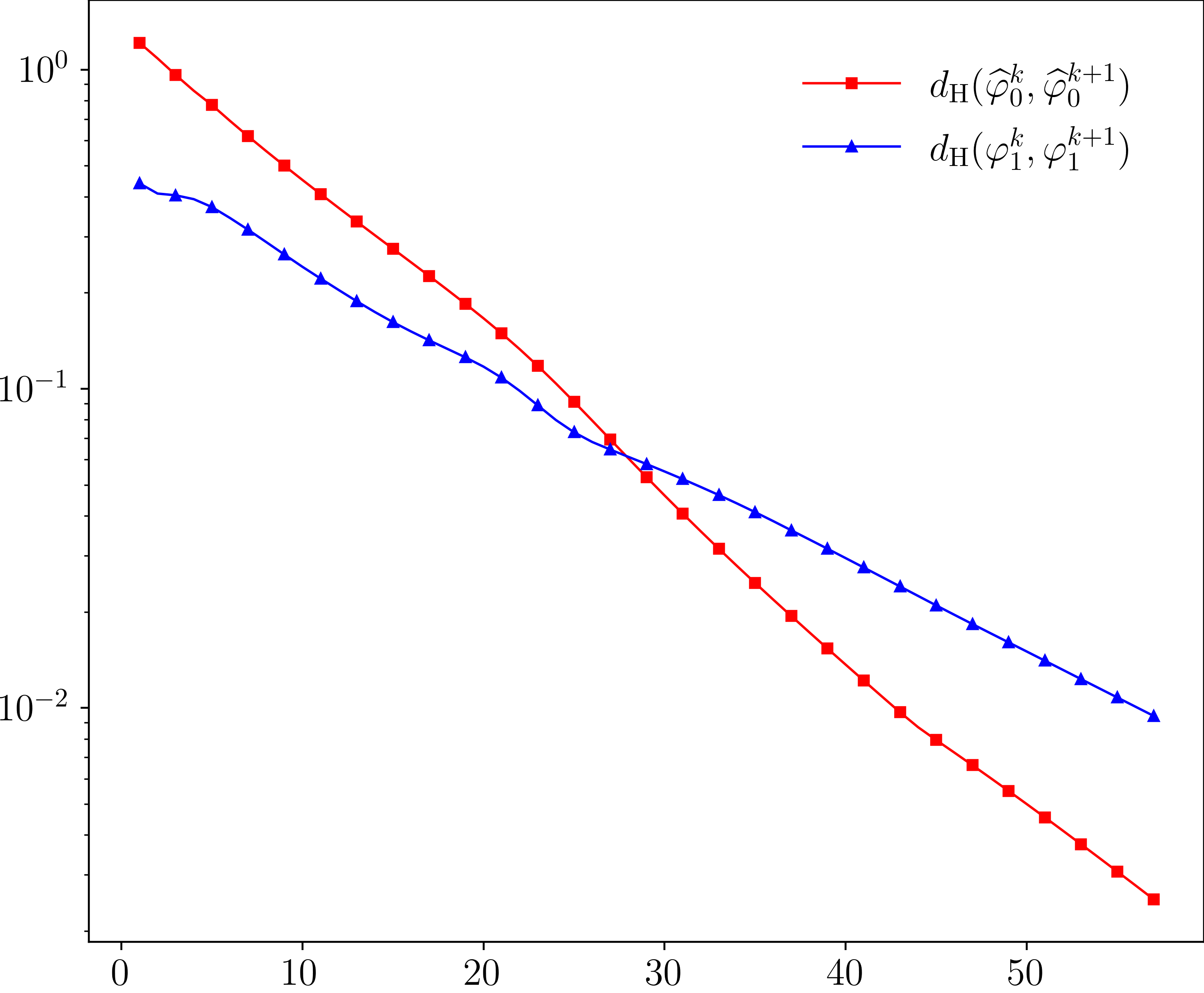}
        \caption{For $\rho_0,\rho_1$ as in Fig. \ref{fig:EndpointPDFs}(a).}
        \label{fig:sub1}
    \end{subfigure}%
    \hfill
    \begin{subfigure}[b]{0.45\linewidth}
        \centering
        \includegraphics[width=\linewidth]{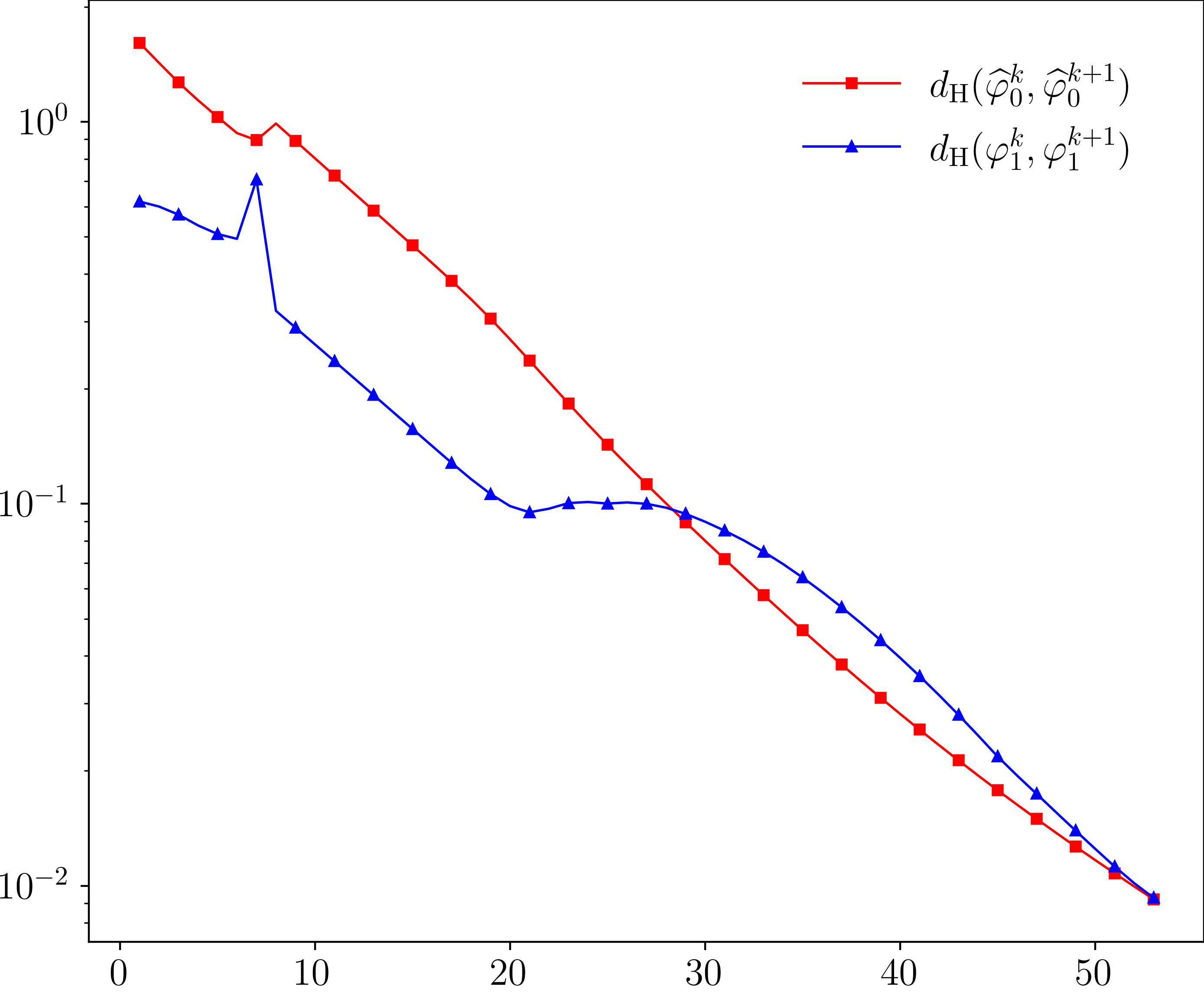}
        \caption{For $\rho_0,\rho_1$ as in Fig. \ref{fig:EndpointPDFs}(b).}
        \label{fig:sub2}
    \end{subfigure}

    \caption{Convergence of the proposed Algorithm \ref{alg:memhorn} for the numerical example in Sec. \ref{sec:NumericalExample}.}
    \label{fig:ConvergenceHilbert}
\end{figure*}

For implementing Algorithm \ref{alg:memhorn} to this problem data, we fix the numerical tolerance $\varepsilon=10^{-2}$ and $\texttt{maxiter}=200$. We perform the backward time marching for the nonlinear PDE IVP \eqref{varphiPDE} via a forward-in-time central-in-space (FTCS) finite difference solver with spatial resolution {{$\Delta x_1 = \Delta x_2 = 2\times10^{-2}$ for the domain $[-1,1]^2$, and temporal resolution $\Delta t = 5\times10^{-5}$}} for the time horizon $[0,1]$. We then forward time march the PDE \eqref{varphihatPDE} by substituting the \emph{most recent backward pass solution of \eqref{varphiPDE}}, and then applying FTCS finite difference {{with the same resolution}}. We then update the boundary condition using \eqref{BilinearBC}, and repeat. 

The resulting convergence with respect to the Hilbert metric $d_{\mathrm{H}}$ is shown in Fig. \ref{fig:ConvergenceHilbert}.

\begin{figure*}[h]
    \centering
    \begin{subfigure}[b]{\linewidth}
        \centering
        \includegraphics[width=0.9\linewidth]{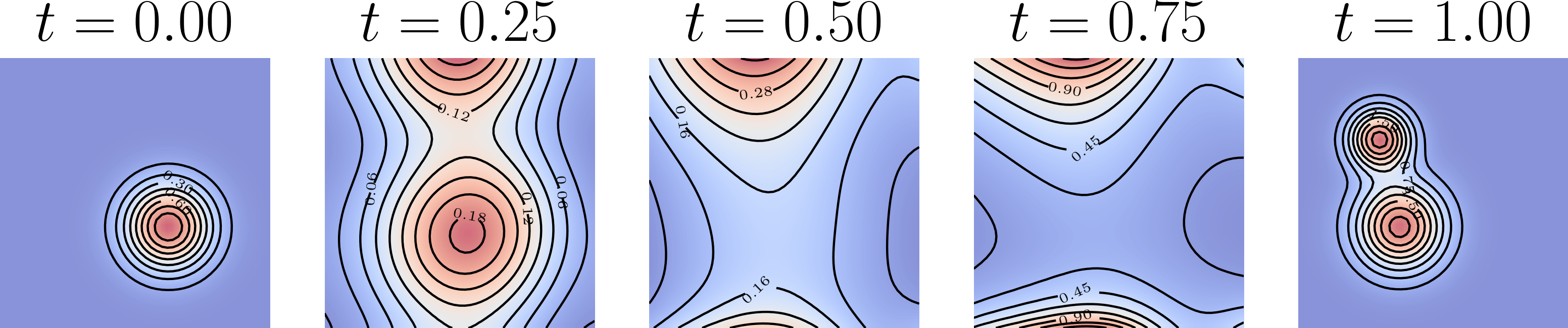} 
        \label{fig:top_s1}
    \end{subfigure}
\vspace{-0.1in} 

\begin{subfigure}[b]{\linewidth}
        \centering
        \includegraphics[width=0.9\linewidth]{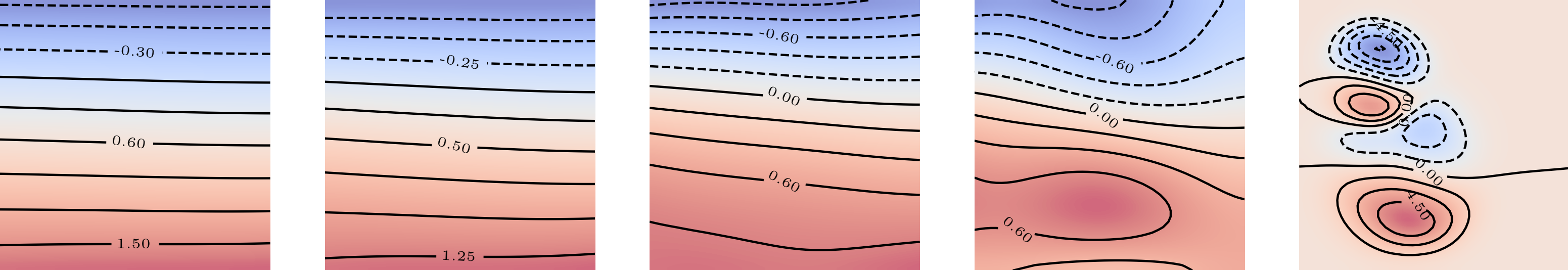}
        \label{fig:bottom_s1}
    \end{subfigure}
    
    \caption{Snapshots of $\rho^{\bm{u}}_{\mathrm{opt}}$ (top panel) and $\bm{u}_{\mathrm{opt}}$ (bottom panel) for the numerical example in Sec. \ref{sec:NumericalExample} with endpoint PDFs as in Fig. \ref{fig:EndpointPDFs}(a). All subplots are over the state domain $[-1,1]^2$.}
    \label{fig:PDFcontrolSim1}
\end{figure*}
\begin{figure*}[h]
    \centering
    \begin{subfigure}[b]{\linewidth}
        \centering
        \includegraphics[width=0.9\linewidth]{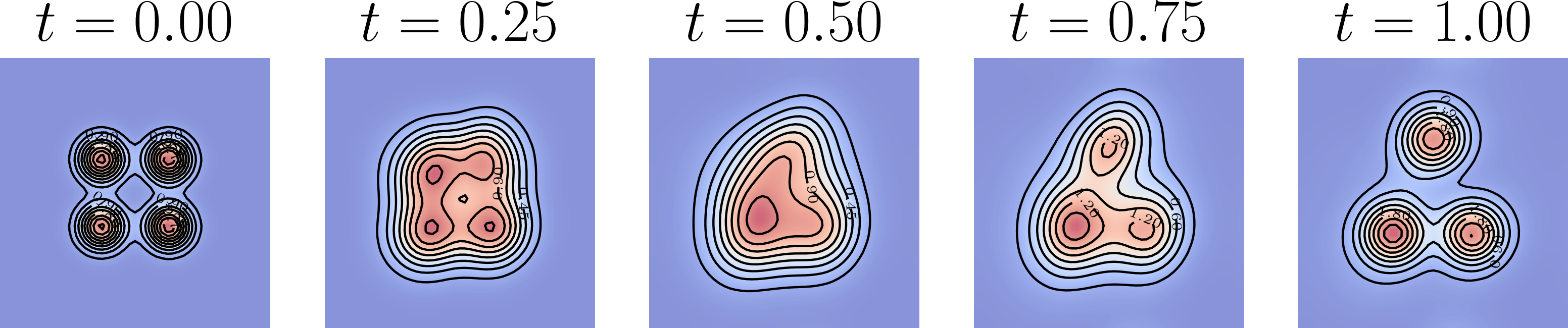} 
        \label{fig:top_s2}
    \end{subfigure}
\vspace{-0.1in} 

\begin{subfigure}[b]{\linewidth}
        \centering
        \includegraphics[width=0.9\linewidth]{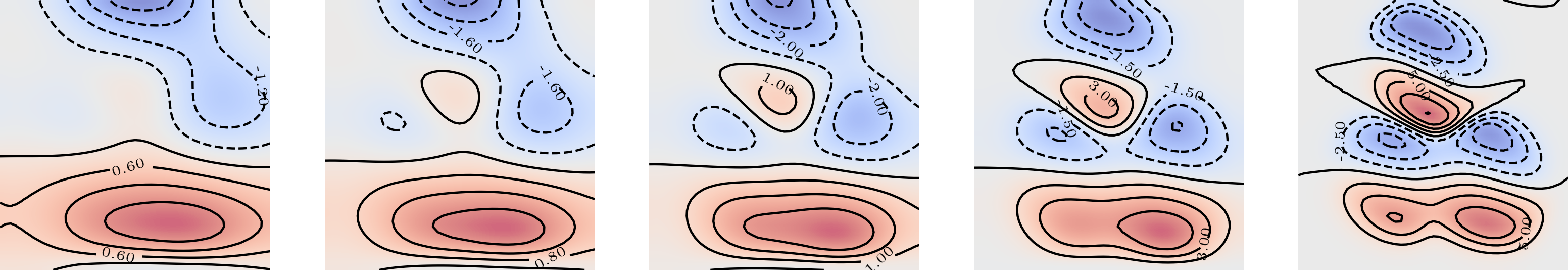}
        \label{fig:bottom_s2}
    \end{subfigure}
    
    \caption{Snapshots of $\rho^{\bm{u}}_{\mathrm{opt}}$ (top panel) and $\bm{u}_{\mathrm{opt}}$ (bottom panel) for the numerical example in Sec. \ref{sec:NumericalExample} with endpoint PDFs as in Fig. \ref{fig:EndpointPDFs}(b). All subplots are over the state domain $[-1,1]^2$.}
    \label{fig:PDFcontrolSim2}
\end{figure*}

The corresponding solution pairs $\left(\rho^{\bm{u}}_{\mathrm{opt}},\bm{u}_{\mathrm{opt}}\right)$ are shown in Fig. \ref{fig:PDFcontrolSim1} and Fig. \ref{fig:PDFcontrolSim2}.

\begin{figure*}[ht!]
\centering
\includegraphics[width=0.9\linewidth]{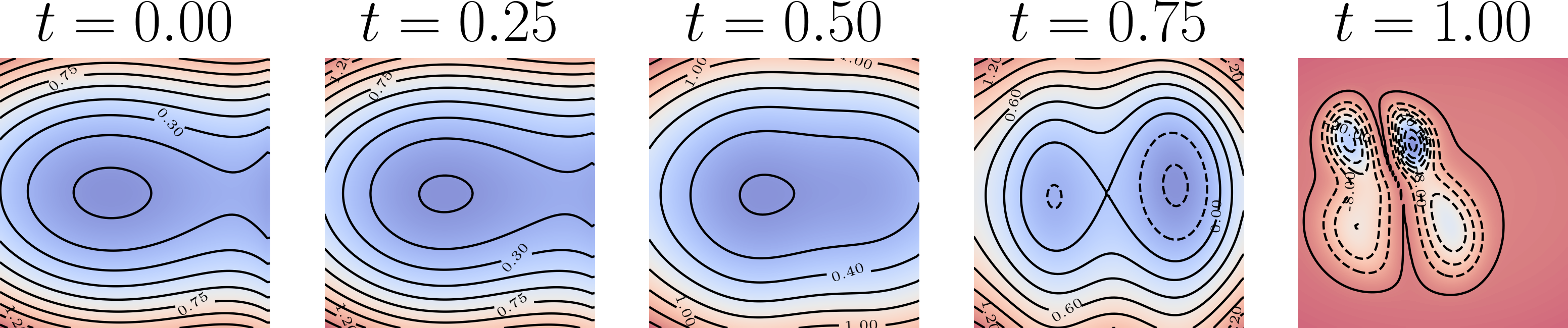}
\caption{Snapshots of $q_{\varphi} + q/\lambda$ with $\lambda=1$ for the numerical example in Sec. \ref{sec:NumericalExample} with endpoint PDFs as in Fig. \ref{fig:EndpointPDFs}(a). All subplots are over the state domain $[-1,1]^2$.}
\label{fig:SumOfReactionRates}
\end{figure*}

Fig. \ref{fig:SumOfReactionRates} shows the snapshots of $q_{\varphi} + q/\lambda$ (here $\lambda=1$) for the converged solution. This particular quantity appears as a sufficient condition for non-expansiveness of $\mathcal{B}$ in Proposition \ref{prop:Bnonexpansive}, and for our numerical example, has the form
$$q_{\varphi} + q = \frac{1}{2}\left(\nabla_{\bm{x}}\log\varphi\right)^{\top}\begin{bmatrix}
-1 & 0\\
0 & 0
\end{bmatrix}\nabla_{\bm{x}}\log\varphi + \frac{1}{2}\bm{x}^{\top}\begin{bmatrix}
1 & 0\\
0 & 2
\end{bmatrix}
\bm{x}.$$
From Fig. \ref{fig:SumOfReactionRates}, this quantity, during the backward pass from $t=1$ to $t=0$, becomes nonnegative after initial transients, i.e., is nonnegative for most but not for all times. This implies there is room for future improvement for our theoretical guarantees, and that the practical performance of the proposed Algorithm \ref{alg:memhorn} is better than the guarantees proved herein. An improved analysis will require new techniques and will be pursued in our future work.


\section{Conclusions}\label{sec:Conclusions}
This work proposes a computational algorithm for solving the generic control-affine Schr\"{o}dinger bridge problem with input and noise channel mismatch. This algorithm enables computational synthesis of optimal feedback controller for nonlinear non-Gaussian density steering over a fixed finite horizon. The proposed algorithm, referred to as ``Sinkhorn with memory," can be seen as an extension of the dynamic Sinkhorn recursion -- the latter is not applicable in the channel mismatch case. We provide a local stability guarantee for the proposed algorithm. This guarantee takes the form of convergence to a Hilbert metric ball around the fixed point with a radius estimate in terms of the channel mismatch magnitude. Numerical examples highlight that the proposed algorithm performs well in practice. 

\section*{References}

\bibliographystyle{IEEEtran}
\bibliography{References.bib}

\end{document}